\documentclass[12pt]{article}
\usepackage{amsfonts,amssymb,color,url}
\usepackage{longtable}
\usepackage[all,knot,arc]{xy}
\usepackage{cite}
\usepackage{graphicx}
\usepackage{caption}
\usepackage{subcaption}

\newcommand{\rank}{\mathop{\mathrm{rank}}\nolimits}

\newtheorem{theorem}{Theorem}[section]
\newtheorem{prop}[theorem]{Proposition}
\newtheorem{lemma}[theorem]{Lemma}
\newtheorem{cor}[theorem]{Corollary}
\newtheorem{unnumber}{}

\newtheorem{prepf}[unnumber]{Proof}
\newenvironment{proof}{\prepf\rm}{\endprepf}
\newcommand{\qed}{\hfill$\Box$}

\newtheorem{preex}[unnumber]{Example}
\newenvironment{example}{\preex\rm}{\endpreex}
\newtheorem{prerk}[unnumber]{Remark}
\newenvironment{remark}{\prerk\rm}{\endprerk}
\newtheorem{preprob}{Problem}
\newenvironment{problem}{\preprob\rm}{\endpreprob}

\newcommand{\gf}{\mathop{\mathrm{GF}}}

\newcommand{\agl}{\mathop{\mathrm{AGL}}}
\newcommand{\asl}{\mathop{\mathrm{ASL}}}
\newcommand{\psl}{\mathop{\mathrm{PSL}}}
\newcommand{\psu}{\mathop{\mathrm{PSU}}}
\newcommand{\pgamu}{\mathop{\mathrm{P\Gamma U}}}

\newcommand{\pgl}{\mathop{\mathrm{PGL}}}
\newcommand{\pgaml}{\mathop{\mathrm{P}\Gamma\mathrm{L}}}
\newcommand{\agaml}{\mathop{\mathrm{A}\Gamma\mathrm{L}}}

\newcommand{\M}{\mathord{\mathrm{M}}}
\newcommand{\Co}{\mathord{\mathrm{Co}}}
\newcommand{\Sz}{\mathop{\mathrm{Sz}}}
\newcommand{\HS}{\mathord{\mathrm{HS}}}
\newcommand{\Sp}{\mathop{\mathrm{Sp}}}
\newcommand{\GF}{\mathop{\mathrm{GF}}}
\newcommand{\pxl}{\mathop{\mathrm{PXL}}}
\newcommand{\pgu}{\mathop{\mathrm{PGU}}}

\newcommand{\asigl}{\mathop{\mathrm{A}\Sigma\mathrm{L}}}

\newcommand{\Aut}{\mathop{\mathrm{Aut}}}
\newcommand{\Out}{\mathop{\mathrm{Out}}}
\newcommand{\End}{\mathop{\mathrm{End}}}

\begin{document}

\title{The Existential Transversal Property: a Generalization of Homogeneity and its Impact on Semigroups}
\author{Jo\~ao Ara\'ujo\footnote{Departamento de Matem\'atica, Faculdade de Ci\^encias e Tecnologia (FCT)  
Universidade Nova de Lisboa (UNL), 2829-516 Caparica, Portugal;   jj.araujo@fct.unl.pt} \footnote{Universidade Aberta, R. Escola Polit\'{e}cnica, 147, 1269-001 Lisboa, Portugal}
 \footnote{CEMAT-Ci\^{e}ncias, 
Faculdade de Ci\^{e}ncias, Universidade de Lisboa
1749--016, Lisboa, Portugal; jjaraujo@fc.ul.pt} , Wolfram Bentz\footnote{School of Mathematics \& Physical Sciences, University of Hull, Kingston upon Hull, HU6 7RX, UK; W.Bentz@hull.ac.uk}, and Peter J. Cameron\footnote{School of Mathematics and Statistics, University of St Andrews, North Haugh, St Andrews, Fife KY16 9SS, UK; pjc20@st-andrews.ac.uk}}
\date{}
\maketitle

\begin{abstract}
Let $G$ be a permutation group of degree $n$, and $k$ a positive integer with
$k\le n$. We say that $G$ has the \emph{$k$-existential property}, or
\emph{$k$-et}, if there exists a $k$-subset $A$ (of the domain
$\Omega$) whose orbit under $G$ contains  transversals for all $k$-partitions $\mathcal{P}$ of $\Omega$.

It is known that for $k< 6$ there are several families of $k$-transitive groups, but for $k\ge 6$ the only ones are alternating or symmetric groups. The first goal of this paper is to show that in the $k$-et context the threshold is $8$, that is, for $8\le k\le n/2$, the only groups with $k$-et are the symmetric and alternating groups; this is best possible. We then (almost) determine the groups with $k$-et for
$4\le k \le n/2$. All these considerations essentially answer an open problem and were the linchpin for the following theorem on semigroups: \emph{Let $G \le S_n$ ($n\ge 27$) be any group (with one exception) having $k$-et and let $B$ witness it (for $k\le n/2$). Then $\langle G,t\rangle$ is regular for all transformation $t$ such that $\Omega t=B$ if and only if $k\le3$ or G is intransitive or the orbit of any $(k-1)$-set has  transversals for all $(k-1)$-partitions.} 
\end{abstract} 

\section{Introduction}

In \cite{ArCa}, the first and third author investigated the $k$-\emph{universal transversal
property}, or $k$-ut property for short, of a permutation group $G$ on
$\Omega$. We say that $G$ has this property if, given any $k$-subset $S$ of
$\Omega$ (subset with $k$ elements), and any $k$-partition $\mathcal{P}$ of $\Omega$
(partition with $k$ parts), there is an element $g\in G$ which maps
$S$ to a section (or transversal) for $\mathcal{P}$. The paper comes close to giving
a characterisation of permutation groups with this property, for
$2<k< \lceil n/2\rceil$, together with several applications to semigroup theory. 

The aim of this paper is to tackle Problem 5 of \cite{ArCa}, the study of the $k$-\emph{existential transversal property}, or $k$-et property, a concept  that is much weaker than the $k$-ut property, and so its consequences for semigroups are
substantially stronger. We say that the permutation group $G$ has the
$k$-et property if there exists a $k$-subset $S$ of $\Omega$ such that, for any
$k$-partition $\mathcal{P}$ of $\Omega$, there is an element $g\in G$ which maps $S$ to a
transversal (or cross-section)  for $\mathcal{P}$. The first part of our goal is to understand groups with
the $k$-et property for $k\le n/2$.

Recall that a permutation group $G$ is \emph{$k$-homogeneous} if it acts
transitively on the set of all $k$-subsets of $\Omega$. We have the
obvious implications
\[\hbox{$k$-homogeneous}\Rightarrow\hbox{$k$-ut}\Rightarrow\hbox{$k$-et}.\]

The first theorem in the paper of Livingstone and Wagner~\cite{lw} asserts that
a $k$-homogeneous group of degree $n$, with $k\le\lceil n/2\rceil$, is
$(k-1)$-homogeneous. A significant result of the earlier paper is a theorem
in the same spirit: for $2<k<\lceil n/2\rceil$, $k$-ut implies $(k-1)$-ut.
A similar result for the $k$-et property does not hold.  Indeed,
there are exactly two counterexamples, which are very interesting, and one of the
surprising results  of the paper.

In the second section of this paper, we introduce the $k$-et property, with
a little background, and prove a number of results about it. The following
theorem summarises the main results of this section.

\begin{theorem}
\begin{enumerate}\itemsep0pt
\item Intransitive groups with $k$-et for $2<k<n$ are known.
\item Transitive groups with $k$-et for $3<k<n-2$ are primitive.
\item Transitive groups with $k$-et, $5\le k\le n/2$ are $2$-homogeneous for sufficiently large $n$.
\end{enumerate}
\end{theorem}
In the last statement, we remark that the restriction on $n$ being sufficiently large will be eliminated later in the article.

In the next section, we present several examples of the $k$-et property,
and show that some of the results in the preceding theorem are best possible.

The next two sections tackle the classification problem. In Section~\ref{s:8}
we show:

\begin{theorem}
For $8\le k\le n/2$, a permutation group with the $k$-et
property is symmetric or alternating.
\end{theorem}

The theorem is best possible. The Mathieu group $M_{24}$  has the $k$-et property  for $k\le 7$
but not for $k=8$. The following section shows that it is the only $7$-et group
apart from symmetric and alternating groups, and also gives a complete
classification of $6$-et groups, and nearly complete classifications for
$k$-et groups with $k=4,5$.

The techniques developed in these sections allow some improvements to be made
in the results of \cite{ArCa}; we turn to this in Section~\ref{s:ut}, and also
correct a few small mistakes in that paper (a gap in the  proof of \cite[Proposition 2.6]{ArCa} and a couple of
missing groups in \cite[Theorem 4.2(4)]{ArCa}).

{\color{black}
After this, we turn to the applications for semigroups, which provided the
motivation for this group theory problem. We are concerned with semigroups of the form
$\langle G,t\rangle$, where $G$ is a permutation group on $\Omega$ and $t$ a
transformation of $\Omega$ which is not a permutation. Our main interest is
in regularity: an element $x$ of a semigroup $S$ is regular if it has a von
Neumann inverse $x'$ (satisfying $xx'x=x$), and a semigroup is regular if all
its elements are regular. The basic result, due to Levi, McAlister and 
McFadden~\cite{lmm}, asserts that $t$ is regular in $\langle G,t\rangle$ if and only if
there exists $g\in G$ such that $tgt=t$. Such an element $g$ maps the image
of $t$ to a transversal for the kernel of $t$. Hence we see that
\begin{itemize}\itemsep0pt
\item every map $t$ of rank $k$ is regular in $\langle G,t\rangle$ if and only
if $G$ has the $k$-universal transversal property;
\item every map $t$ with image $B$ satisfying $|B|=k$ is regular in
$\langle G,t\rangle$ if and only if $G$ has the $k$-existential transversal
property with witness $B$.
\end{itemize}

Note that  a non-regular semigroup can be generated by its regular elements;
therefore, the fact that every element in $G$ is regular ($g=gg^{-1}g$) and $t$ is regular in $\langle G,t\rangle$, does not imply that $\langle G,t\rangle$ is regular. 
However, the key result in  \cite{ArCa} 
(asserting that for $k\le n/2$, the $k$-ut implies
the $(k-1)$-ut) ensures that if $G$ has the $k$-ut property  for $k \le n/2$ and $t$ has rank $k$, then
the semigroup $\langle G,t\rangle$ is in fact regular.  Our aim in this paper is to
investigate the much more difficult question: when  is it true that $\langle G,t\rangle$ is regular for all maps whose image is a given $k$-set $B$? 
It is easy to see that if $G$ has the
$k$-et property with witnessing set $B$, and also has the $(k-1)$-ut property, then this is 
true. However,  this sufficient condition is not necessary. In fact, 
with the exception of one sporadic group, we fully solve the regularity problem, in the sense that  our result on groups with unknown  $k$-et status are conditional on them having this necessary property.   
}

The following theorem compiles the main results on regularity of semigroups.

\begin{theorem}
Let $G \le S_n$ be a group different from $\Sz(32):5$ (degree $n=1025$), and $k \le n/2$. Suppose $G$ possesses the $k$-et property and $B$ witnesses it. 
Then the semigroup $\langle G,t\rangle$ is regular for all image $B$ transformations $t \in T_n$ if and only if one of the following holds
\begin{enumerate}
\item $k \le 3$ or $k\ge 7$. 
\item $k=6$ and $G$ possesses $5$-ut, is intransitive, or is one of the following groups: $\pgl(2,17)$ ($n=18$), $\M_{11}$ ($n=12$), $\M_{23}$($n=23$). 
\item $k=5$ and $G$ is not $\pgl (2,27)$ or $\pgaml(2,27)$ ($n=28$). 
\item $k=4$ and $G$ possesses $3$-ut, is intransitive, or is $G=\agl (1,13)$ ($n=13$). 
\end{enumerate}
In particular, if $n\ge27$ and $k\ge 4$, then $\langle G,t\rangle$ is regular if and only if $G$ is intransitive or possesses $(k-1)$-ut.
\end{theorem}
 
 These sets $B$ have many interesting interpretations in terms of finite geometries, but we refer the reader to Section \ref{app}. In a more speculative register, these sets might be connected  with bases (sets of smallest size whose pointwise stabiliser is the identity), but we could not decide the issue.  

The general context of this paper is the following. The theory of transformation semigroups, through its connections to theoretical computer science and automata theory,  quickly led to several very natural problems, which were totally hopeless with the techniques available three or four decades ago. However, given the enormous progress made in the last decades, permutation
{\color{black}group theory} now has the tools to  answer many of those problems. The problems usually  translate into beautiful statements in the language of permutation groups and combinatorial structures, as shown in several recent investigations (for a small sample please see   \cite{AAC,abc,abc2,abcrs,abdkm,arcameron22,ArCa,acmn,acs,ArMiSc,ArnoldSteinberg,randomsynch,gr,neumann,sv15}). One especially interesting consequence of the results in this paper is that they unearth some finite geometric structure on the image  ranges of transformations (apparently acting on unstructured sets). 

\section{The $k$-et property}

Throughout, $G$ denotes a permutation group on a finite set $\Omega$, with
$|\Omega|=n$.

A permutation group $G$ on $\Omega$ has the \emph{$k$-existential transversal
property} if there exists a $k$-subset $S$ of $\Omega$ such that, for any
$k$-partition $\mathcal{P}$ of $\Omega$, there exists $g\in G$ such that $Sg$ is a
section (or transversal) for $\mathcal{P}$. We call $S$ a \emph{witnessing $k$-set}.
We write $k$-et for short.

A useful consequence of $k$-et is the following.

\begin{prop}\label{p:witness}
Suppose that $G$ has the $k$-et property. Then $G$ has at most $k$ orbits on $(k-1)$-sets,
and a witnessing $k$-set contains representatives of every $G$-orbit on
$(k-1)$-sets.
\end{prop}

\begin{proof}
The first statement clearly follows from the second. Let $S$ be the
witnessing $k$-set. If $|A|=k-1$, let $\mathcal{P}$ be the partition with the elements
of $A$ as singleton parts and one part $\Omega\setminus A$. Then, if
$Sg$ is a section for $\mathcal{P}$, we have $Ag^{-1}\subseteq S$.
\qed\end{proof}

We say that $G$ has the \emph{weak $k$-et property} if there exists a
$k$-set containing representatives of every $G$-orbit on $(k-1)$-sets.
One way to show that a permutation group $G$ does not have the $k$-et property
is to show that $G$ is the automorphism group of a structure containing two
$(k-1)$-subsets which cannot be contained in a $k$-set. We will say that two
such subsets \emph{cannot coexist}.

We make two further observations about the weak $k$-et property.

\begin{prop}\label{p:stab}
Suppose that $G$ is transitive, and the stabiliser of a point has the weak
$k$-et property. Then $G$ has the weak $k$-et property.
\end{prop}

\begin{proof}
Let $S$ be a witnessing $k$-set for the point stabiliser $G_x$, and let $A$
be any $(k-1)$-subset of the domain. We can move $A$ by an element of $G$ to
ensure that it does not contain $x$, and move the result into $S$ by an 
element of $G_x$.\qed
\end{proof}

\begin{prop}\label{p:order}
If $G$ has the weak $k$-et property, then $|G|\ge{n\choose k-1}/k$.
\end{prop}

\begin{proof}
Each orbit of $G$ on $(k-1)$-sets has size at most $|G|$, and there are at
most $k$ such orbits.
\qed\end{proof}

Note that (although it is not the case that the $k$-et property implies the
$(k-1)$-et property) if $G$ fails the $k$-et property because the above bound
fails, then $G$ fails the $k'$-et property for
all $k'$ with $k\le k'\le n/2$. This is because the ratio of consecutive
values of the right-hand side is $(n-k+1)/(k+1)$, which is greater than $1$
for $k<n/2$.

\medskip 

We can obtain a slightly better bound if $G$ has the $k$-et property.

\begin{prop}\label{p:order2}
If $G$ has the $k$-et property, then $|G|\ge\frac{2}{k+1}{n\choose k-1}$.
\end{prop}

\begin{proof}
Let $B$ witness the $k$-et property.
Each orbit of $G$ on $(k-1)$-sets has size at most $|G|$, so the bound holds if there are at
most $(k+1)/2$ such orbits. 

Assume that there are $k\ge l > (k+1)/2$ orbits. At least $2l-k\ge 2$ orbits have a unique representative in $B$. Let $B_1, B_2$ 
be two such representatives, and let $b_i$ be the unique element of $B\setminus B_i$. Consider the $k$-partition  $\mathcal{P}$ of $\Omega$ consisting of
$\{b_1,b_2\}, \Omega \setminus B$, and the singleton subsets of $B \setminus \{b_1,b_2\}$. If $g \in G$ is such that $Bg $ is a transversal of $\mathcal{P}$, then
clearly one of $B_1,B_2$, say $B_1$, is contained in $Bg$.

However, $B_1$ is the unique representative of its orbit in $B$, and thus $B_1g=B_1$ and $b_2g \in \Omega \setminus B$. It follows that the stabilizer of $B_1$ is 
non-trivial, and so its orbit has size at most $|G|/2$. Repeating the above argument, we see that all but one of the orbits with unique representatives in $B$ have size 
at most  $|G|/2$. Summing over all orbits we obtain that 
$${n\choose k-1}   \le |G| (l-2l+k) +\left(|G|/2\right)(2l-k-1)+|G|,$$
and the bound follows.
\qed\end{proof}
As with the bound for weak $k$-et, if $G$ fails the order bound for $k$, it does so for all $k'$ 
with $k \le k'\le n/2$. 
Note, however, that the techniques of these theorems do not allow us to
improve Proposition~\ref{p:witness}. The group $\pgl(2,32)$ with degree $33$
turns out to have the $5$-et property, and to have five orbits on $4$-sets.

We will utilise both the stronger bound from Proposition \ref{p:order2}, and the slightly weaker, but simpler, bound from Proposition \ref{p:order}.
By abuse of notation we will refer to both expressions as the \emph{order bound}. 
\medskip

We remark that if a permutation group $G$ preserves a geometric structure, it can often be used to show that $G$ is not  $k$-et. 
A typical arguments runs along the following lines: suppose $G$ preserves at least two tiers of non-trivial 
geometric objects. Then for a partition ${\mathcal P}$ with``cascading" sets, as depicted in Figure \ref{fig:FIG2} for an affine-like geometry, we can often conclude that any section cannot be contained in a 
geometric object. On the other hand, for a partition with $k-1$ singletons ``crammed" into a small flat (see Figure \ref{fig:FIG1}), any section often will need to lie in a flat, potentially from a higher tier. These conditions cannot simultaneously hold for any $k$-set, and so the group does not satisfy $k$-et.

\begin{figure}
\centering
\begin{minipage}{.45\textwidth}
  \centering
  \includegraphics[width=.8\linewidth]{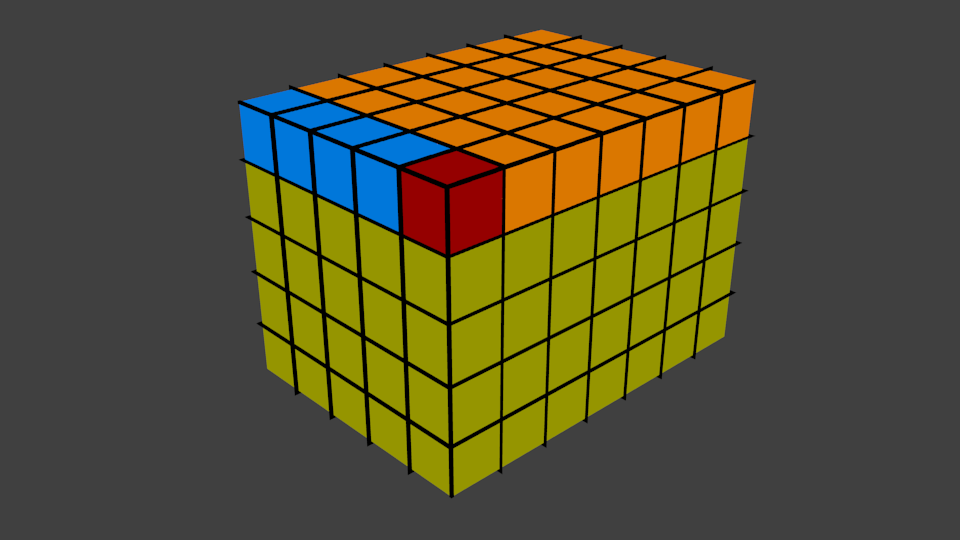}
  \captionof{figure}{A ``cascading" partition. Each part extends the union of all smaller parts to a geometric object of the next higher tier. Under suitable conditions, no section of the partition lies in a small-tier flat.}
  \label{fig:FIG2}
\end{minipage}
\hspace{0.2cm}
\begin{minipage}{.45\textwidth}
  \centering
  \includegraphics[width=.8\linewidth]{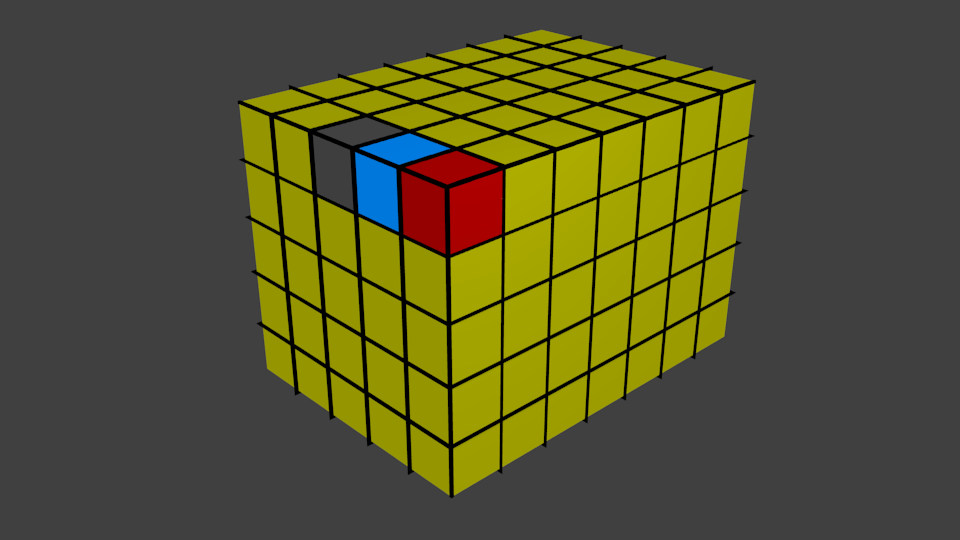}
  \captionof{figure}{A ``crammed" partition, in which $k-1$ singleton parts are placed in a flat from the smallest tier possible. Under suitable conditions, any section lies in a flat from a small tier.}
  \label{fig:FIG1}
\end{minipage}
\end{figure}

Even if $G$ preserves only one tier of geometric objects, we obtain restrictions on the potential witnessing sets for $k$-et. If  every flat is uniquely determined by any $k-2$ of its points, then 
by the weak-$k$-et property, every potential witness for $k$-et contains $k-1$ point within one flat and an additional    point (see Figure \ref{fig:FIG5}). 
In addition the stabiliser of a flat must act $(k-1)$-homogeneously on it, as visualised by the partition in Figure \ref{fig:FIG3}.  If a flat is already determined by $k-3$ points, then the stabiliser of a flat also acts transitively on its complement, as can be seen from Figure  \ref{fig:FIG4}.

\begin{figure}
\centering
\includegraphics[width=.35\linewidth]{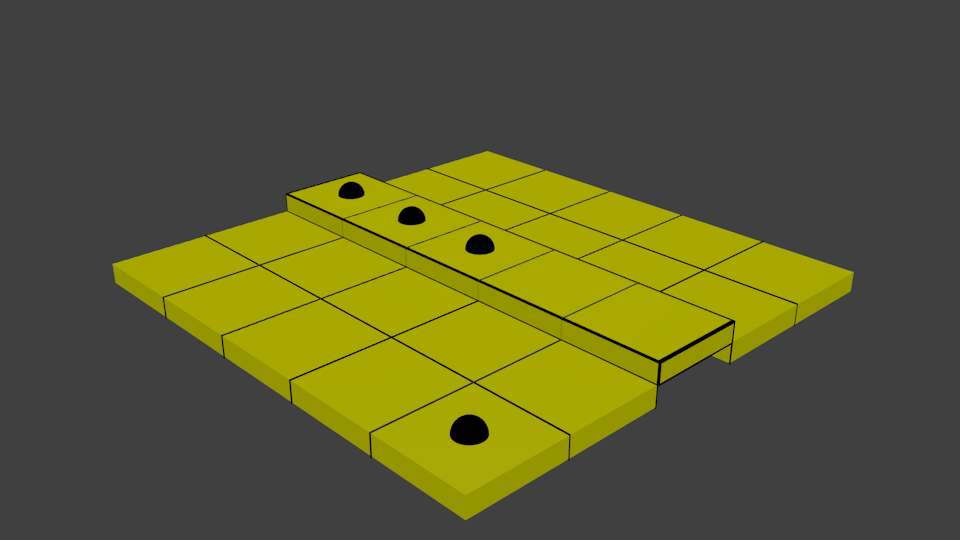}
\captionof{figure}{If a group preserves any type of flat that is determined by $k-1$ of its points, then a witnessing set for $k$-et consists of $k-1$ points in a flat  and an additional point.}
\label{fig:FIG5}
\end{figure}

\begin{figure}
\centering
\begin{minipage}{.45\textwidth}
  \centering
  \includegraphics[width=.8\linewidth]{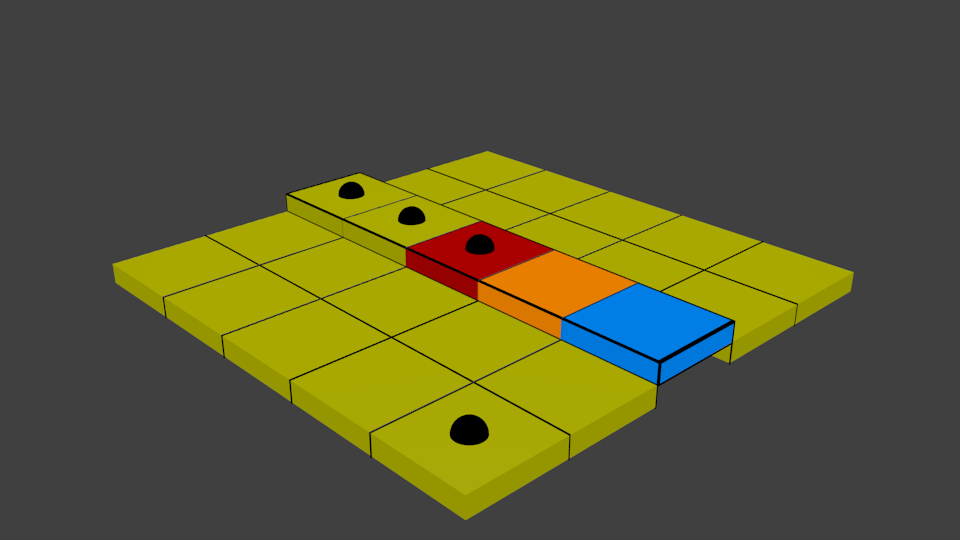}
  \captionof{figure}{A partition showing that the stabiliser of a flat acts $(k-1)$-homogeneously on it.}
  \label{fig:FIG3}
\end{minipage}
\hspace{0.2cm}
\begin{minipage}{.45\textwidth}
  \centering
  \includegraphics[width=.8\linewidth]{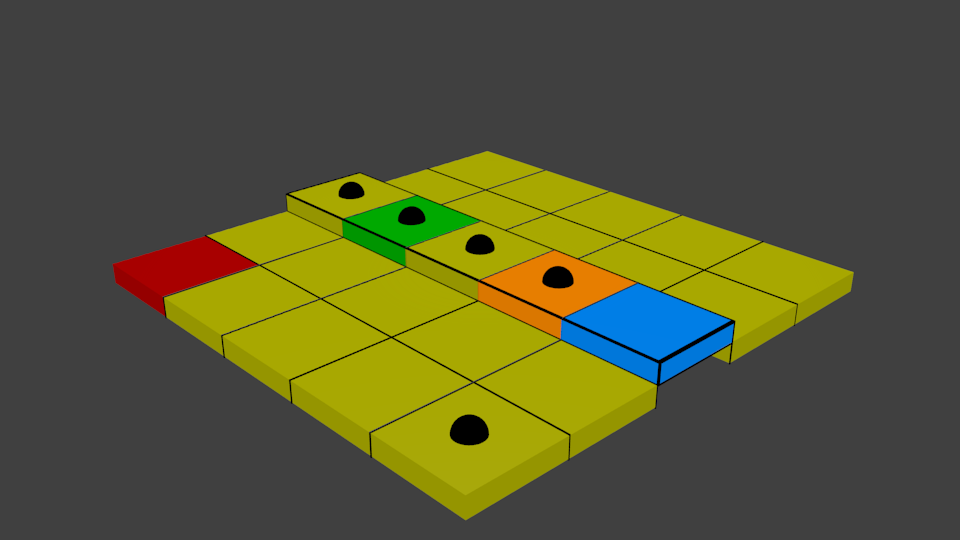}
  \captionof{figure}{{A partition showing that the stabiliser of a flat acts transitively on its complement.}}
  \label{fig:FIG4}
\end{minipage}
\end{figure}

Now we turn to the classification of intransitive groups with $k$-et.

\begin{prop}\label{p:intrans}
Let $G$ be an intransitive permutation group with the $k$-et property,
where $2<k<n$. Then $G$ fixes a point and acts $(k-1)$-homogeneously on
the remaining points.
\end{prop}

\begin{proof} If $G$ fixes a point $x$, then the witnessing $k$-set $K$
must contain $x$. But there is only one $(k-1)$-subset of $K$ which does
not contain $x$, so $G$ must be $(k-1)$-homogeneous on the points different
from $x$.

Suppose that $G$ has two complementary fixed sets $A$ and $B$, with $|A|=a$ and
$|B|=b$, so that $a+b=n$; suppose that $a,b\ge2$. Then there is a $(k-1)$-set
$L_1$ satisfying $|L_1\cap A|=\min(k-1,a)$, and a $(k-1)$-set $L_2$ satisfying
$|L_2\cap B|=\min(k-1,b)$. These two sets must have images fitting inside a
$k$-set; so
\[\min(k-1,a)+\min(k-1,b)\le k.\]
But the left hand side is
\[\min(2k-2,k-1+a,k-1+b,n).\]
Since $a,b\ge2$ and $3\le k\le n-1$, the minimum is at least $k+1$, a
contradiction.
\qed\end{proof}

The converse is also true. This gives a complete characterisation of the
intransitive $k$-et groups for $2<k<n$, and shows that $k$-et implies
$(k-1)$-et for $2<k<n/2$ and intransitive groups.

\begin{prop}\label{prop2.4}
Suppose that $G$ fixes a point and is $(k-1)$-homogeneous on the remaining
points. Then $G$ has the $k$-et property.
\end{prop}

\begin{proof}
Let $a$ be the fixed point, and $B$ any $k$-set containing $a$. We claim that
$B$ witnesses the $k$-et property. Let $\mathcal{P}=\{A_1,\ldots,A_k\}$ be a
$k$-partition, where $A_1$ is the part containing $a$, and choose
$a_i\in A_i$ for $i=2,\ldots,k$. Choose an element $g\in G$ mapping
$\{a_2,\ldots,a_k\}$ to $B\setminus\{a\}$; then $B$ is a section for $\mathcal{P}g$.
\qed\end{proof}

However, the class of intransitive $2$-et groups is larger. Any permutation
group $G$ with two orbits $A$ and $B$, such that $G$ acts transitively on
$A\times B$, has the $2$-et property, with a set containing one point from
each orbit as a witnessing set. For suppose that $G$ has two orbits and is
transitive on their product, and let $\mathcal{P}=\{P_1,P_2\}$ be any $2$-partition.
Without loss of generality, $P_1$ contains a point $a\in A$. If $P_2$ contains
a point $b\in B$, then $\{a,b\}$ is the required section; so we can suppose
that $B\subseteq P_1$. Running the argument the other way, we see that also
$A\subseteq P_1$, a contradiction.

\medskip

Now we turn to transitive groups.
We say that a transitive permutation group $G$ is \emph{fully imprimitive} if
the following equivalent conditions hold:
\begin{enumerate}\itemsep0pt
\item any two points of $\Omega$ are contained in a proper block of
imprimitivity;
\item every orbital graph for $G$ is disconnected;
\item for any $\alpha\in\Omega$ and $g\in G$, $\langle G_\alpha,g\rangle\ne G$.
\end{enumerate}
For example, a group $G$ in its regular action is fully imprimitive if and only
if it is not cyclic.

\begin{prop}\label{p:2et}
The transitive group $G$ has the $2$-et property if and only it is not
fully imprimitive. Moreover, a $2$-set witnesses $2$-et if and only if it is not contained in any proper block of imprimitivity.
\end{prop}

\begin{proof}
The set $S$ is a witnessing $2$-set if and only if the graph $X$ with vertex set
$\Omega$ and edge set $S^G$ has the property that, for every $2$-partition
$\{A,B\}$ of $\Omega$, $X$ has an edge between $A$ and $B$. This simply means
that $X$ is connected. The proposition follows.
\qed\end{proof}

\begin{prop}\label{p:2blocks}
A transitive imprimitive permutation group with the $3$-et property has two
blocks of imprimitivity in any block system; a witnessing set contains two
points from one block and one from the other. Moreover, if $n>4$, the block
system is unique.
\end{prop}

\begin{proof} 
A witnessing set cannot be contained in a block, and cannot contain three
points from distinct blocks; so it must have two points from one block and
one from another. But if $\mathcal{P}$ is a partition two of whose three parts are
blocks, then no image of $S$ is a transversal for $\mathcal{P}$. So there can be at
most two blocks.

If there are two block systems, then blocks from the two systems intersect in
$n/4$ points; this contradicts the previous paragraph unless $n=4$.
\qed\end{proof}

\begin{example}
The obvious place to look is for a maximal imprimitive group with
two blocks of imprimitivity. So take the wreath product $S_m\wr S_2$, with
$n=2m$. We claim that a $3$-set containing two points
from one block of imprimitivity witnesses $3$-et.

Take any $3$-partition $\{P_1,P_2,P_3\}$. If none of the three parts meets
both bipartite blocks, then two of them are contained in one block and one
in the other, and the assertion is clear. So suppose that $P_3$ meets both
bipartite blocks. Choose arbitrary representatives for $P_1$ and $P_2$. If
they happen to be in the same block, then choose the representative for $P_3$
from the other block; otherwise choose any representative. Note that the example holds if we replace $S_m$ with another $2$-homogeneous group.
\end{example}

It is not the case that primitive groups with $3$-et are $2$-homogeneous
with finitely many exceptions.

\begin{example}
There is an infinite family of primitive groups which have the $3$-et
property but are not $2$-homogeneous.

The groups we take are $S_n\wr S_2$ in its product action on the square grid
of size $n$, for $n\ge4$. We claim that $S=\{(1,1),(1,2),(2,3)\}$ witnesses the
$3$-et property.

Suppose that we have any $3$-partition $\{P_1,P_2,P_3\}$ of the grid. First we observe that there
is a line of the grid meeting at least two parts. For if every horizontal line
is contained in a single part, then any vertical line meets all three parts.

Suppose first that $L$ is a line meeting only $P_1$ and $P_2$, without loss
$L=\{(1,x):x\in\{1,\ldots,n\}\}$. Let $A_i=\{x:(1,x)\in P_i\}$ for $i=1,2$.
If $a_i\in A_i$ for $i=1,2$, then every point $(j,x)$ with $x\ne a_1,a_2$ and
$j>1$ must lie in $P_1\cup P_2$, since otherwise we would have a transversal
in the orbit of $S$. So for any point $(u,v)\in P_3$, we have $u>1$ and
$v=a_1$ or $v=a_2$. Assuming without loss that $|A_2|>1$, we can repeat the
argument with another point of $A_2$; this leads to the conclusion that
$|A_1|=1$ and $P_3\subseteq\{(j,a_1):j>1\}$. Without loss, take $a_1=1$.

Suppose that $(2,1)\in P_3$. If there is a point $(x,y)\in P_2$ with $x\ne 2$
and $y\ne 1$, then we have a section in the orbit of $S$; so suppose not,
so that every such point lies in $P_1$. Again, this forces that $P_3$ is a
singleton $\{(2,1)\}$, and all points $(x,y)$ with $x>2$ and $y>1$ lie in
$P_1$. If any point $(x,1)$ with $x>2$ belongs to $P_1$, then we have our
section $\{(1,2),(2,1),(x,1)$; so we can suppose that all these points belong
to $P_2$. But then $\{(2,1),(3,1),(4,2)\}$ is the required section.

The other case is that no line meets just two parts of the partition. Choose
a line $L$ meeting all three parts, say $L=\{(1,x):x\in\{1,\ldots,n\}\}$.
Let $A_i=\{x:(x,1)\in P_i\}$. Now we find a section of the required kind if,
for example, $\{2,\ldots,n\}\times A_1$ contains a point of $P_1$; so we may
assume that this set is contained in $P_2\cup P_3$, and similarly for the 
other two sets of this form. Since $n\ge4$, at least one of the sets $A_i$
has size greater than $1$, say $A_1$. Now choose $a,a'\in A_1$. Then the
line $L=\{(a,x):x\in\{1,\ldots,n\}\}$ meets $P_1$ and one other part, so it
meets all three parts. If $(a,x_2)\in P_2$ and $(a,x_3)\in P_3$, then these
two points together with $(a',1)$ form the required section.
\end{example}

Transitive groups with $k$-et for $k>3$ must be primitive:

\begin{prop}\label{p:prim}
Let $G$ be a transitive permutation group of degree $n$ having the $k$-et
property, where $n>9$ and $3<k<n-2$. Then $G$ is primitive.
\end{prop}

\begin{proof}
Suppose that the permutation group $G$ has the $k$-et property, and is
transitive and
imprimitive, with $b$ blocks of imprimitivity of size $a$. Let $l=k-1$.
If $K$ is a witnessing $k$-set, then $K$ contains a representative of every
orbit of $G$ on $l$-sets, and hence contains an $l$-set partitioned into
at most $b$ parts each of size at least $a$ in every possible way by the
blocks of imprimitivity. Hence any two such partitions differ in a single
move (consisting of reducing one part by one and increasing another by one).
So the question becomes:
\begin{quote}
For which $l,a,b$ is it true that any two partitions of $l$ into at most $b$
parts of size at most $a$ differ by at most one move?
\end{quote}
We will call such partitions \emph{admissible}.
When we phrase the problem in this way, we see that it is invariant under
replacement of $l$ by $n-l$, where $n=ab$; so we may assume, without loss
of generality, that $l\le n/2$. Also, replacing a partition by its dual, we
see that it is invariant under interchange of $a$ and $b$; so we may assume
that $b\le a$.

Suppose first that $b=2$. Then $l\le a$, so two allowable partitions are
$(l)$ and $(\lfloor l/2\rfloor, \lceil l/2\rceil)$. These differ by at least
two moves if $l\ge4$. So only $l\le 3$ is possible here, giving $k\le 4$.

If $k=4$, let $A$ be a witnessing set for $4$-et. $A$ contains $3$ elements from one block of imprimitivity  and $1$ element from the other block. However, 
such a set fails $4$-et with any 
$4$-partition in which $2$ partition each form a block of imprimitivity, for a contradiction.

Now suppose that $b\ge3$. For the first subcase, suppose that $l\le a$. Then
the partition $(l)$ is admissible; and there is an admissible partition with
largest part $\lceil l/b\rceil$, and these are at least two steps apart
unless $l=2$. In the second subcase, there is a partition with largest part
$a=n/b$, and a partition with largest part $\lceil l/b\rceil<n/(2b)+1$. If
these are at most one step apart, then $n/b<n/(2b)+1$, giving $n<4b$, so
that $a\le 3$. Since, by assumption, $b\le a$, we have $n\le 9$.

So the theorem is proved.
\qed\end{proof}

The condition $k>3$ is necessary here, as we saw earlier.
Higher $k$ implies higher transitivity:

\begin{prop}\label{p:2homog}
Let $k\ge 5$, and $G$ be transitive of degree $n$ with the $k$-et property,
where $n>R(k-1,k-1)$. Then $G$ is $2$-homogeneous.
\end{prop}

\begin{proof}
We know that $G$ is primitive. If it is not $2$-homogeneous, it has more than
one orbit on $2$-sets. Partition these orbits into two parts, called red and
blue, in any manner. Since $n>R(k-1,k-1)$, Ramsey's theorem implies that there is a
monochromatic $(k-1)$-set, say red; the witnessing $k$-set $S$ must contain a red
$(k-1)$-set. So the blue edges within $S$ form a star, and all but
(at most) one of them have valency $1$. So the blue graph cannot have adjacent
vertices of degree greater than $1$, since such a configuration would give us a
triangle or path of length~$3$ in the representative $k$-set. But this is a
contradiction, since the blue graph is regular and connected.
\qed\end{proof}

We will in fact show that all such transitive groups with $5\le k\le n/2$ are $2$-homogeneous.
By a different argument we can extend this result to the $4$-et property.

\begin{theorem} \label{t:4et2hom} Let $G$ be a permutation group of degree $n\ge 8$ that satisfies $4$-et.
If $G$ is primitive, but not $2$-homogeneous, then $n=100$ and $G$ is either the Higman-Sims group or its automorphism group.
\end{theorem}

\begin{proof}
Suppose that $G$ is such a group, of degree $n\ge8$.  Let $\Gamma$ be any
orbital graph for $G$, and let $A$ be a $4$-set witnessing $4$-et. Then all
$3$-vertex induced subgraphs of $\Gamma$ are represented in $A$. Since a
$3$-clique and a $3$-coclique cannot coexist within a $4$-set, we see that
either $\Gamma$ or its complement is triangle-free.

Suppose that $\Gamma$ is triangle-free. Then $\Gamma$ must contain all other
$3$-vertex graphs. (By Ramsey's theorem it must contain a null graph of size
$3$. If it omits the graph with three vertices and two edges then it consists
of isolated edges, and if it omits the graph with three vertices and one edge
then it is complete bipartite. In either case, $G$ is imprimitive.) Then we see
that the induced subgraph on $A$ must be a path of length $2$ and an isolated
vertex.

So the witnessing $4$-set contains two or four edges of any orbital graph,
which implies that $G$ has at most three orbits on $2$-sets. First we eliminate
the case when there are three orbits. In this case, let the orbital graphs be
$\Gamma_1$, $\Gamma_2$, $\Gamma_3$. All three are triangle-free, and the
structure of $A$ is as follows, up to choice of numbering: $\{1,2\}$ and
$\{2,3\}$ are edges of $\Gamma_1$; $\{1,4\}$ and $\{2,4\}$ of $\Gamma_2$; and
$\{1,3\}$ and $\{3,4\}$ of $\Gamma_3$. So the end-vertices of a path of
length~$2$ in $\Gamma_i$ are joined in $\Gamma_{i-1}$ (indices mod~$3$).

Let the valencies of the three graphs be $k_1$, $k_2$, and $k_3$, and suppose
without loss of generality that $k_1\ge k_3$. Now there are $k_1(k_1-1)$ paths
of length two in $\Gamma_1$ leaving a vertex $v$; all end among the $k_3$
neighbours of $v$ in $\Gamma_3$. So if $w\in\Gamma_3(v)$, the number of common
neighbours of $v$ and $w$ in $\Gamma_1$ is $k_1(k_1-1)/k_3\ge k_1-1>k_1/2$.
So any two neighbours $w,w'$ of $v$ in $\Gamma_3$ have a common neighbour in
$\Gamma_1$, giving a triangle with two edges $\{v,w\}$ and $\{v,w'\}$ in
$\Gamma_1$ and one edge $\{w,w'\}$ in $\Gamma_2$, a contradiction.

So we can assume that $G$ has just two orbits on $2$-sets, and is a group of
automorphisms of a triangle-free strongly regular graph. Using CFSG, the only
such graphs are known ones with $10$, $16$, $50$, $56$, $77$ or $100$ vertices;
computation shows that only the last of these has automorphism group with the
$4$-et property (and moreover, both the Higman--Sims group and its automorphism
group have this property).  \qed
\end{proof}

\section{Some examples}
\label{s:examples}

In this section, we treat various families of groups. First, the groups
$\agl(d,2)$.

\begin{theorem}\label{t:ex.affine}
Let $G$ be the affine group $\agl(d,2)$, where $d\ge3$. Then $G$ has the $k$-et
property for $k\le 4$,
but fails the $k$-et property for all $k$ with $4<k\le n/2$ except when $d=4$.
The group $\agl(4,2)$ has the $6$-et property but not the $5$-et property. 
\end{theorem}

\begin{remark} 
This example shows that (unlike the $k$-ut property) the $k$-et property is not
monotone.
\end{remark}

\begin{proof}
The affine group $\agl(d,2)$ is $3$-transitive, and so certainly has the
$k$-et property for $k\le 3$. We will show in Theorem~\ref{egregious} that it
does not have the $k$-et property for $k\ge8$. We treat the remaining values
individually.

\paragraph{The case $k=4$:}
We claim that an affine independent $4$-set witnesses the $4$-et property.
For let $\{P_1,\ldots,P_4\}$ be any $4$-partition, with (without loss)
$|P_4|>1$. Choose arbitrary representatives of $P_1,P_2,P_3$. There is at
most one point which makes an affine plane with the three chosen
representatives; since $|P_4|>1$ we can avoid this point in our choice of the
representative for $P_4$.

\paragraph{The case $k=5$:}
There are two orbits on $5$-sets, an affine plane with an extra
point, and an affine independent $5$-tuple. To defeat the first type, take
the partition $\mathcal{P}=(P_1,\ldots,P_5)$, where $P_1$ and $P_2$ are singletons,
$P_3$ is a $2$-set extending $P_1\cup P_2$ to a plane, $P_4$ a $4$-set
extending $P_1\cup P_2\cup P_3$ to a $3$-space, and $P_5$ the remaining set.
To defeat the second type, let $\mathcal{P}$ be the partition where $P_1,\ldots,P_4$
are singletons forming an affine plane and $P_5$ the remaining set.

\paragraph{The case $k=6$:} First we deal with $\agl(4,2)$.
Let $G=\agl(4,2)$. Then $G$ is $3$-transitive, has
two orbits on $4$-sets (planes, and affine-independent sets), two orbits on
$5$-sets (plane plus point and affine-independent), and three orbits on
$6$-sets (six points of a $3$-space, an affine-independent $5$-set and one
point making a plane with three points of this set, and a $6$-set all of whose
$5$-subsets are affine independent. Call these orbits on $6$-sets $A$, $B$,
$C$. We are going to show that a set of type $B$ is a witnessing set. Note that
given an affine independent $5$-set, there are ${5\choose3}=10$ points which
enlarge it to a set of type $B$, and only one which enlarges it to a set of
type $C$.

Let $\mathcal{P}$ be a $6$-partition.

\subparagraph{Case 1:} there are four singleton parts forming a plane. Then
given any fifth point, there are just three points which enlarge the resulting
$5$-set to a $6$-set of type $A$. Since the remaining two parts of the
partition, say $U$ and $V$, have twelve points between them, at least one
(say $V$) has size greater than $3$. So take a point of $U$, and then we can
find a point of $V$ which enlarges the resulting $5$-set to a $6$-set of
type $B$, as required.

\subparagraph{Case 2:} not the above.
We claim first that from any four parts of $\mathcal{P}$ we can choose representatives
not forming a plane. If all four parts are singletons, this is true by the case
assumption; otherwise, choose representatives of three of the parts excluding
one part of size bigger than $2$, then all but one point of the final part
will work.

Now the six parts of $\mathcal{P}$ contain $16$ points, so the largest two parts, say
$U$ and $V$, together contain at least six points. Choose representatives
of the other four parts not forming a plane. Then just four points extend
this four-set to an affine dependent set, so some point of $U\cup V$ extends
the given four-set to an affine independent $5$-set.

If $|U|$ and $|V|$ are each at least two, then we can find a point in the
unused part which extends the $5$-set to a $6$-set of type $B$, since only
one point fails to do this.

In the remaining case, the partition has five singleton parts (forming an
affine independent $5$-set) and one part containing everything else. But all
but one point of this part extends the $5$-set to a $6$-set of type $B$. We
are done.

\medskip

For $d>4$, there is an additional type of $6$-set, namely an affine independent
set. Now a $6$-partition constructed as before (consisting of the differences
in an increasing chain of subspaces) has the property that any transversal
must be affine independent. On the other hand, a partition with four singleton
parts forming an affine plane has the property that all its transversals are
affine dependent. So $\agl(d,2)$ does not have $6$-et for $d>4$.

\paragraph{The case $k=7$:}
For $k=7$, $d\ge5$, there is an affine independent $6$-set, and a $6$-set
contained in an affine $3$-space. These two sets cannot be contained in a 
common $7$-set. For $d=4$, we can replace the first set with a set of type $C$.
\qed\end{proof}

Now we turn to the largest Mathieu group, and show:

\begin{theorem}\label{t:m24}
$M_{24}$ has the $k$-et property for $k\le 7$ (but not for larger $k$).
\end{theorem}

\begin{proof}
Let $G$ be the Mathieu group $M_{24}$, with $n=24$.
This group is $5$-transitive, so it has the $k$-et property for $k\le5$.

We show the $6$-et property. Recall that $M_{24}$ is the
automorphism group of a Steiner system $S(5,8,24)$ (blocks of size $8$, any
five points in a unique block). We claim that a $6$-set not contained in a
block is a witnessing set (these form a single orbit of $M_{24}$). For take
any $6$-part partition. By the Pigeonhole Principle, one of its parts (without
loss $P_6$) contains at least four points. Choose arbitrary representatives
of $P_1,\ldots,P_5$. These representatives lie in a unique block, which
has at most three more points; so there is a point of $P_6$ not in this block;
choose this as the representative of $P_6$.

A similar but more intricate argument shows that $M_{24}$ has the $7$-et
property. (It is a property of the Steiner system that, of any seven points,
six of them are contained in a block; a witnessing set is one in which six but
not all seven points lie in a block.) We have confirmed this by computer, and
also showed that it fails to have the $8$-et property. \qed
\end{proof}

We remark in passing that these sets of size $7$ are the minimal bases for
the permutation group $M_{24}$ (sets of smallest size whose pointwise
stabiliser is the identity).

We now turn to a collection of groups which have the $4$-et property.

\begin{theorem}\label{t:ex.Sp}
Let $G$ be the symplectic group $\Sp(2d,2)$ (in one of its $2$-transitive
representations of degree $2^{2d-1}\pm2^{d-1}$, or the affine symplectic group
$2^{2d}:\Sp(2d,2)$ of degree $2^{2d}$ (with $d\ge2$), or the
Conway group $\Co_3$ of degree $276$. Then $G$ has the $4$-et property.
\end{theorem}

\begin{proof}
We treat these groups using a variant of the arguments of
\cite[Proposition 4.7]{ArCa}. In each case the group has just two orbits
on $3$-subsets, each orbit $T$ forming a \emph{regular two-graph}
\cite{taylor}: this means
\begin{enumerate}\itemsep0pt
\item any $4$-set contains an even number of members of $T$;
\item any two points lie in $k$ members of $T$.
\end{enumerate}
The values of $k$ for the groups of interest are:
\begin{itemize}\itemsep0pt
\item For $\Sp(2d,2)$ with $n=2^{2d-1}\pm2^{d-1}$, $k=2^{2d-2}$ or
$k=2^{2d-2}\pm2^{d-1}-2$.
\item For $2^{2d}:\Sp(2d,d)$ with $n=2^{2d}$, $k=2^{2d-1}$ or
$2^{2d-1}-2$.
\item For $\Co_3$, $k=112$ or $162$.
\end{itemize}

In connection with (a), we will call a $4$-set \emph{full}, \emph{mixed},
or \emph{empty} according as it contains $4$, $2$ or $0$ members of $T$. It
is clear from Proposition~\ref{p:witness} that the only possible witnessing
sets for the $4$-ut property are the mixed sets.

We need a small amount of theory of regular two-graphs. Suppose that $T$
is a regular two-graph. For any point $x$, form a graph with vertex set
$\Omega\setminus\{x\}$ whose edges are all pairs $\{y,z\}$ for which
$\{x,y,z\}\in T$. We say this graph is obtained by \emph{isolating} $x$.
Now the graph uniquely determines $T$: a triple $\{a,b,c\}$ containing $x$
is in $T$ if and only if the two vertices different from $x$ form an edge;
and a triple not containing $x$ is in $T$ if and only if an odd number of
$\{a,b\}$, $\{b,c\}$ and $\{a,c\}$ are in $T$. Also, the graph is regular
with valency $k$.

\begin{lemma}
Suppose that $G$ is an automorphism group of a regular two-graph $T$. Suppose
that
\begin{enumerate}\itemsep0pt
\item $G$ is transitive on the set of mixed $4$-sets;
\item $(n-4)/3<k<2(n-1)/3$.
\end{enumerate}
Then $G$ has the $4$-et property, with the mixed $4$-sets as witnessing sets.
\end{lemma}

\begin{proof}
Suppose that there is a partition $\mathcal{P}=\{P_1,\ldots,P_4\}$ (with
$|P_1|\le\cdots\le|P_4|$) for which no mixed $4$-set is a section. Thus every
section to $\mathcal{P}$ is a full or empty $4$-set. We first show that either all
sections are full, or all are empty. Suppose that $\{x_1,\ldots,x_4\}$ is a
full $4$-set with $x_i\in P_i$ for $i=1,\ldots,4$. If $x'_4$ is another
point in $P_4$, then $\{x_1,x_2,x_3,x'_4\}$ is a section containing a
$3$-set $\{x_1,x_2,x_3\}\in T$, so it must be a full $4$-set. By connectedness,
every section is full.

By replacing the two-graph by its complement if necessary, we may assume that
the sections are all full.

\subparagraph{Case 1:} $|P_1|=1$. Let $P_1=\{a\}$, and let $\Gamma$ be the
graph obtained by isolating $a$. Then $\Gamma$ contains all the edges of the
complete tripartite graph with tripartition $\{P_2,P_3,P_4\}$; so a vertex in
$P_2$ is joined to everything in $P_3\cup P_4$, and its valency is at least
$2(n-1)/3$ (since $P_2$ is the smallest of these three parts),
contradicting (b).

\subparagraph{Case 2:} $|P_1|>1$. Choose $a,b\in P_1$. Consider the $4$-set
$\{a,b,p_2,p_3\}$, where $p_2\in P_2$ and $p_3\in P_3$. Since
$\{a,p_2,p_3\}$ and $\{b,p_2,p_3\}$ are in $T$, we see that both or neither
of $\{a,b,p_2\}$ and $\{a,b,p_3\}$ are in $T$. Again by connectedness, it
follows that either $\{a,b,q\}\in T$ for all $q\notin P_1$, or this holds
for no $q\notin P_1$. Hence, in the graph obtained by isolating $a$, either
$b$ is joined to all vertices not in $P_1$, or to none of them. Thus either
$k\ge 3n/4$ or $k\le n/4 - 2$, contradicting (b).

This proves that every partition has a mixed $4$-set as a section. By (a),
a mixed $4$-set witnesses the $4$-et property.
\qed\end{proof}

Now we turn to the proof of the theorem. Note that these groups are all
$2$-transitive, and so have the $k$-et property for $k\le 2$; they
were shown to have the $3$-ut property (and hence the $3$-et property) in
\cite{ArCa}. (This also follows, more easily, from Proposition~\ref{p:watkins}
below.)

The group $\Sp(4,2)$ on $6$ points is $S_6$ and clearly has the
$4$-et property. Excluding this case, each of these groups is, as noted, the
automorphism group of a regular two-graph; so we only have to verify the
hypotheses of the Lemma. For (b), this is simple arithmetic; so we need to
prove that $G$ is transitive on mixed $4$-sets. For $\Co_3$, this can
be checked by computation.

For the infinite families, we argue as follows. We show that the groups in
question have $6$ orbits on $4$-tuples of distinct points whose underlying
set is a mixed $4$-set. This will prove the claim, since there are six ways
of selecting two $3$-subsets of a $4$-set.

Our main tool is \emph{Witt's Theorem}, see \cite[Theorem 7.4]{classical}.
We can translate any $4$-set so that it contains $0$, and show that the
triples of points making up a mixed $4$-set with $0$ fall into six orbits.
Witt's theorem says that if $f:U\to V$ is a linear isometry on a subspace $U$
of a formed space $V$ with radical $\mathrm{Rad}(V)$, and $f$ maps
$U\cap\mathrm{Rad}(V)$ to $f(U)\cap\mathrm{Rad}(V)$, then $f$ extends to a
linear isometry from $V$ to $V$. In our case, the radical of $V$ is $\{0\}$,
so the second condition is automatically satisfied.

In the case $G=2^{2d}:\Sp(2d,2)$, the space $V$ will be the
$2d$-dimensional space over $\GF(2)$ with a symplectic form $B$ on it.
In the case $G=\Sp(2d,2)$ in either of its $2$-transitive
actions, $\Omega$ can be identified with the set of zeros of a non-singular
quadratic form of one of the two possible types on the space $V$ (which also
carries a symplectic form $B$, obtained by polarising the quadratic form). We
will apply Witt's theorem to these formed spaces. In each case, the triples
of the two-graph can be taken as those $\{x,y,z\}$ for which
\[B(x,y)+B(y,z)+B(z,x)=0.\]

First note that a mixed $4$-set cannot be a subspace $W$ of $V$. For if so,
then either the symplectic form restricted to $W$ is identically zero, or it
is the unique such form on a $2$-dimensional space. Thus, either $B(x,y)=0$
for all $x,y\in W$, or $B(x,y)=1$ for all distinct non-zero $x,y\in W$;
calculation shows that $W$ is empty or full in the two cases.

So, if $\{0,a,b,c\}$ is a mixed $4$-set, then $\{a,b,c\}$ is a basis for a
$3$-dimensional subspace $U$ of $V$. Of the three inner products $B(a,b)$,
$B(b,c)$ and $B(c,a)$, one or two are zero, so there are six possibilities. 
The values of $B$ on basis vectors determine uniquely its values on the whole
of $U$.

In the case of $\Sp(2d,2)$, we also have a quadratic form $Q$, which is
zero on $\{0,a,b,c\}$, and we see that the values of $Q$ on $U$ are also
determined, by the polarisation rule
\[Q(x+y)=Q(x)+Q(y)+B(x,y).\]
Hence there are just six orbits of the group on such tuples, as claimed.
\qed\end{proof}

\medskip

Another group which is an automorphism group of a regular two-graph is the
Higman--Sims group, with degree $176$. This group was shown in \cite{ArCa} to
have the $3$-ut property. We do not know whether it has $4$-et, but it is
possible to show that it has weak $4$-et.

\medskip

Here is an example to show that $k$-et does not imply $(k-1)$-ut for all
but finitely many groups.

\begin{theorem}\label{t:ex.psl}
Let $q$ be a prime power, and $\psl(3,q)\le G\le\pgaml(3,q)$. Then $G$ has
the $4$-et property but not the $3$-ut property.
\end{theorem}

\begin{proof}
$G$ acts $2$-transitively on the point set $\Omega$ of the projective plane
$\mathrm{PG}(2,q)$. The group induced on a line of the plane by its setwise
stabiliser contains $\pgl(2,q)$, and so is $3$-transitive; and the pointwise
stabiliser of the line contains the translation group of the affine plane,
and so is transitive on the complement of the line. Thus, $G$ has just two
orbits on triples (collinear and noncollinear triples), and is transitive
on $4$-tuples $(x_1,x_2,x_3,x_4)$ where $x_1,x_2,x_3$ lie on a line $L$ and
$x_4\notin L$.

We show first that $G$ does not satisfy $3$-ut. (This is a special case of an
argument in \cite{ArCa}.) Let $a$ be a point on a line $L$, and consider the
partition $\{\{a\}, L\setminus\{a\}, \Omega\setminus L\}$. Clearly any
section consists of three noncollinear points.

Now we show that the set $\{x_1,\ldots,x_4\}$ in the first paragraph is a
witnessing set for the $4$-et property. Let $\{P_1,P_2,P_3,P_4\}$ be any
$4$-partition of $\Omega$.

First we show that there is a line $L$ meeting at least three parts of the
partition. Let $L'$ be any line. If $L$ meets at least three parts, then 
take $L=L'$. If not, suppose without loss that $L'\subseteq P_1\cup P_2$.
Choose $y_3\in P_3$ and $y_4\in P_4$, and let $L$ be the line $y_3y_4$.
Then $L$ intersects $L'$, and so contains a point in either $P_1$ or $P_2$.

Now let $L$ be a line meeting at least three parts. If $L$ meets only
three parts, say $P_1,P_2,P_3$, choose $y_i\in P_i\cap L$ for $i=1,2,3$
and $y_4\in P_4$; if $L$ meets all four parts, then choose any point
$y_4\notin L$, and suppose without loss that $y_4\in P_4$, and then
choose $y_i\in P_i\cap L$ for $i=1,2,3$. In either case, $(y_1,\ldots,y_4)$
is a section for the partition and lies in $(x_1,\ldots,x_4)G$. \qed
\end{proof}

\section{The $k$-et property for $k\ge8$}
\label{s:8}

In this section we show that there is an absolute bound on $k$ for which a
transitive $k$-et group other than a symmetric or alternating group can exist.
Our result is as follows.

\begin{theorem}\label{egregious}
For $8\le k\le n/2$, a transitive permutation group of degree $n$ which has the
$k$-et property is the symmetric or alternating group.
\end{theorem}

The theorem is best possible: we saw earlier that $M_{24}$ has the $7$-et
property. We show in the next section that it is the only such example.

\begin{proof}
Let $G$ be a transitive group of degree $n$ with the $k$-et property, where
$n$ and $k$ are as above. By Proposition~\ref{p:prim}, $G$ is primitive.

We begin with an observation that will be used repeatedly in the proof.
The $k$-et property is closed upwards; so we may assume that $G$ is
a maximal subgroup of $S_n$ other than $A_n$, or a maximal subgroup of $A_n$.

We also need a technique which helps deal with groups which are not
$2$-homogeneous (and which can be adapted to other cases as well).

\begin{lemma}
Let the transitive group $G$ be contained in the automorphism group of a graph
$\Gamma$ with clique number $\omega$ and independence number $\alpha$, and
suppose that $G$ has the $k$-et property, with $k\ge4$.
\begin{enumerate}\itemsep0pt
\item If $\omega,\alpha\ge3$ then $k\ge\omega+\alpha-1$.
\item If $\omega\ge3$ and $\alpha=2$, then $k\ge\omega+2$.
\end{enumerate}\label{l:clique_ind}
\end{lemma}

\begin{proof}
Note that $G$ is primitive, by Lemma \ref{p:prim}.

(a) Suppose that $G$ has the $k$-et property with $4<k\le\omega+\alpha-2$.
Choose $l,m$ with $3\le l,m\le k-1$ and $l+m=k+2$. Now choose two $(k-1)$-subsets
$A$ and $B$ such that $A$ contains an $l$-clique $C$ and $B$ contains an
independent set $D$ of size $m$. We show that (in the terminology introduced
earlier) $A$ and $B$ cannot coexist. 

Let $K$ be a $k$-set containing
$G$-images of $A$ and $B$; without loss, $A,B\subseteq K$, and so certainly
$C,D\subseteq K$. Since $|C|+|D|=k+2$, and $|C\cup D|\le k$, we have
$|C\cap D|\ge2$. But this is a contradiction, since two points of $C$ are
joined while two points of $D$ are not.

(b) Suppose that $G$ has the $k$-et property with $4<k<\omega+2$. Let $S$ be
a witnessing $k$-set. Then $S$ contains a $(k-1)$-clique; so the complementary
graph restricted to $S$ is a star. Also, since the complementary graph has
edges, and has valency at least~$2$, it contains a $4$-vertex path or
cycle, and so $S$ must contain such a path or cycle in the complement, a
contradiction.
\qed\end{proof}

We also make frequent use of Proposition~\ref{p:order}, the \emph{order bound},
and the remark following it (asserting that if $G$ is shown to fail $k$-et
because it fails the order bound, then $G$ does not have $l$-et for
$k\le l\le n/2$).

\medskip

Now we begin our analysis of primitive groups. The strategy is almost always
to find a lower bound for $k$ using the Lemma above, by finding a suitable
graph on which our group acts, and showing that for this value of $k$ the
order bound is violated. The calculations for the last step are exceedingly
messy, but in virtually every case we succeed with plenty to spare. (In
outline, a group with $k$-et has order not much less than $n^{k-1}$; but
in all cases we know, or have good upper bounds for, $|G|$.) In the
first case, we outline the calculations.

\paragraph{Case 1:} $G$ is not basic. Then $G\le S_q\wr S_m$, acting on 
the set $\{1,\ldots,q\}^m$ of all $m$-tuples over an alphabet of size $q$.
We may assume that $q>5$, since if $q\le4$ then $G$ has a regular normal
subgroup and is contained in an affine group (this case is treated later).
This group is the automorphism group of the graph in which two tuples are
joined if they agree in at least one coordinate. This graph has a clique of
size $q^{m-1}$, consisting of all $m$-tuples with a fixed value in the first
coordinate; and an independent set of size $q$, consisting of the ``diagonal''
tuples $(x,x,\ldots,x)$ for $x\in\{1,\ldots,q\}$. Thus, if $G$ is 
$k$-et with $k>4$, then $k>q^{m-1}+q-1$. But for $k=q^{m-1}+q-1$, we can
see that the order bound fails.

For we can take $k$ a little smaller, say $k=q^{m-1}+1$; so suppose that 
\[(q!)^mm!=|G|\le{q^m\choose q^{m-1}}/(q^{m-1}+1).\]
The left-hand side is smaller than $q^{qm}m^m$, whereas the right-hand side is
greater than 
\[(q^m-q^{m-1})^{q^{m-1}}/((q^{m-1})^{q^{m-1}})=(q-1)^{q^{m-1}}.\]
This is certainly false for $m>2$.
For $m=2$, we need to do the argument with a little more care. It is enough
to use the exact value $k=q^{m-1}+q-1=2q-1$ given by our argument.

\medskip

We conclude that $G$ must be basic. By the O'Nan--Scott Theorem
\cite[Theorem 4.1A]{dixon}, $G$ is affine, diagonal, or almost simple.

\paragraph{Case 2:} $G$ is diagonal. Then $G\le T^d(\Out(T)\times S_d)$ for
some finite simple group $T$, with $n=|T|^{d-1}$, where $\Out(T)$ is the outer
automorphism group of $T$; and we may assume that equality holds.

We use the fact that outer automorphism groups of simple groups are small.
Certainly, since every simple group $T$ is generated by two elements, we have
$|\Aut(T)|\le|T|^2$, so $|\Out(T)|\le|T|$.

The domain for $G$ is identified with $T^{d-1}$; and $G$ is generated by
right translations, the map 
\[\lambda_t:(t_1,\ldots,t_{d-1})\mapsto
(t^{-1}t_1,\ldots,t^{-1}t_{d-1})\]
for $t\in T$, automorphisms of $T$ acting componentwise (where inner 
automorphisms are represented by the composition of $\lambda_t$ and right
multiplication by $t$ in each coordinate), coordinate permutations, and the
map
\[\sigma:(t_1,\ldots,t_d)\mapsto(t_1^{-1},t_1^{-1}t_2,\ldots,t_1^{-1}t_d).\]

Consider first the case $d=2$. We have $n=|T|$ and $|G|\le2|T|^3$. The order
bound would give
\[2n^3\ge{n\choose k-1}/k;\]
for $k\ge5$, this implies $n\le 240$. So only $A_5$ and $\psl(2,7)$ need
further consideration. But in each case the outer automorphism group has
order $2$, so the left-hand side can be improved to $4n^2$, and the bound
becomes $n\le\sqrt{480}$, which is false for both groups. So $k$-et
fails for $k\ge5$. (In fact, both groups fail $4$-et as well, since they have
respectively $13$ and $30$ orbits on $3$-sets.)

Now consider the general case.

The subgroup $T^{d-1}$ of $G$ acts regularly, so we choose a Cayley graph
for this subgroup which is invariant under $G_1$. Note that the $G_1$-orbit
of a tuple $(t,1,\ldots,1)$, for $t\ne1$, consists of tuples having either a
single non-identity element, or all of its elements equal; we use the set of
all such elements as our connection set. There is a clique of size $|T|-1$
consisting of all elements with a single non-identity entry in the first
coordinate; an independent set of size $3$ is easily constructed. So
$k\ge|T|+1$, which is large enough to violate the order bound if $d$ is not
too large compared to $|T|$ (say $d<|T|$).

In the remaining case we use a similar argument, considering elements which
have at most $d/3$ non-identity coordinates and their images, which have
at least $1+2d/3$ coordinates equal. This time we can produce a clique of size
${\lfloor d/3\rfloor\choose\lfloor d/6\rfloor}(|T|-1)^{\lfloor d/6\rfloor}$,
consisting of elements with $\lfloor d/6\rfloor$ non-identity coordinates
within a fixed $\lfloor d/3\rfloor$-set; again we can build a coclique of
size~$3$, and the order bound is violated.

\paragraph{Case 3:} $G$ is affine. Again we may assume that $G=\agl(d,p)$ for
some $d,p$.

If $k\le d+1$,
there is an affine independent $(k-1)$-set; if $k\ge6$, there
exist five points contained in an affine space of dimension $2$ or $3$.
These
cannot both be contained in a $k$-set. So $k$-et fails. Thus we may assume
that $k\ge d+2$.

If $k\le p^{d-1}$, there is an affine space contained in a hyperplane, and
another  with the property that any hyperplane misses two of its points (take
$d+2$ points, any $d+1$ independent). So we may assume that $k>p^{d-1}$.

Now calculation shows that ${p^d\choose k-1}/k$ is greater than $|G|$,
with finitely many exceptions (indeed, only $\agl(4,2)$ and $\agl(5,2)$
don't satisfy this inequality).

\paragraph{Case 4:} $G$ is almost simple.

The \emph{base size} of a permutation group is the smallest number of points
of the domain whose pointwise stabiliser is the identity.
By results of Tim Burness with various co-authors (see \cite{bls}), an
almost simple primitive group $G$ satisfies one of the following:
\begin{enumerate}\itemsep0pt
\item $G$ is a symmetric or alternating group, acting on subsets of fixed size
or uniform partitions of fixed shape;
\item $G$ is a classical group, acting on an orbit on subspaces or
complementary pairs of subspaces of the natural module;
\item the base size of $G$ is at most $7$, with equality only in the case
$G=M_{24}$, $n=24$.
\end{enumerate}

\subparagraph{Case 4(a):} $G$ is $S_m$ on $r$-sets or uniform $r$-partitions.

First consider the case that $G$ acts on $r$-sets, with $m>2r$. Form a graph by
joining two $r$-sets if their intersection is non-empty. There is a clique of
size $m-1\choose r-1$ consisting of all $r$-sets containing a specified point,
and an independent set of size $\lfloor m/r\rfloor$ consisting of pairwise
disjoint $r$-sets. If $m\ge3r$, the Lemma applies, and shows that
$k\ge{m-1\choose r-1}$, and the order bound is violated.

If $2r<m<3r$, use instead the graph where two $r$-sets are joined if they
intersect in $r-1$ points. There is a clique of size $m-r+1$ consisting of
all $r$-sets containing a fixed $(r-1)$-set, and an independent set of size
$\lceil(m-r+2)/2\rceil$ consisting of $r$-sets intersecting pairwise in a
given $(r-2)$-set. So $k\ge m-r+\lceil(m-r+2)/2\rceil$, and again the order
bound is violated.

\medskip

Now consider the case that $G$ acts on partitions with $r$ parts of size $s$,
with $rs=m$. Let $p(r,s)$ be the number of such partitions; so
\[p(r,s)=(rs)!/(s!)^rr!.\]
If $r>2$, make a graph by joining two partitions which have a part in common.
There is a clique of size $p(r-1,s)$ and a large coclique, so the usual 
argument works.

Suppose that $r=2$. Join two partitions if their common refinement has two
parts of size $1$. There is a clique of size $s+1$ containing all partitions
for which one part contains a given $(s-1)$-set. To produce a large independent
set, if $s$ is even, take a partition of $\{1,\ldots,m\}$ into $s$ parts of
size $2$, and consider partitions which are unions of parts in this
subsidiary partition. If $s$ is odd, leave two isolated points, and put one
into each part of the partition made up of parts of the subsidiary partition.

\subparagraph{Case 4(b):} $G$ is a classical group on an orbit of subspaces
or pairs of subspaces of complementary dimension in its natural module; in
the latter case we may assume that either the subspaces are complementary
or one contains the other, and for groups preserving a form we may assume
that subspaces are either totally singular or non-singular.

We defer the cases $G=\pgaml(2,q)$ (with $n=q+1$) and
$\mathrm{P}\Gamma\mathrm{U}(3,q)$ (with $n=q^3+1$) until Case 4(c) below.

Suppose first that $G=\pgaml(m,q)$ on $1$-dimensional subspaces (with
$n=(q^m-1)/(q-1)$, $m\ge3$). A similar argument to that
used for the affine groups applies. If $m+1<k\le(q^{m-1}-1)/(q-1)+1$, then a
$(k-1)$-subset of a hyperplane, and a $(k-1)$-set containing $m+2$ points with
no $m+1$ in a hyperplane, cannot be moved inside the same $k$-set by $G$;
so we may assume that $k>(q^{m-1}-1)/(q-1)+1$, and the order bound is violated.

In the case $\Gamma\mathrm{L}(m,q)$ on $r$-dimensional subspaces, we may
assume that $1<r\le m/2$. We follow the argument for $S_m$ on $r$-sets: if
$r\le m/3$, join two subspaces if they intersect; if $r>m/3$, join them if
their intersection is a hyperplane in each.

For other classical groups on subspaces, an almost identical approach works,
except in the case of split orthogonal groups $O^+(2r,q)$ acting on totally
singular $r$-spaces. In this case the graph given by ``intersection of
codimension $1$'' is bipartite, so we take codimension~$2$ instead.

For groups acting on pairs of subspaces, we can join two pairs if one subspace
in the pair coincides. To find a coclique, use the fact that the incidence
matrix of $r$-spaces and $(m-r)$-spaces is invertible
(Kantor~\cite{kantor:inc}), so there is a bijection
between the two sets of subspaces such that each subspace is incident with
its image (contains it, or is contained in it). 

\subparagraph{Case 4(c):} $G$ has base size at most $6$. (We can ignore
$M_{24}$, since computation shows that this group fails the $k$-et property for 
$8\le k\le 12$.) In this case we know little about the structure of $G$, so
we proceed differently.

First we make a couple of observations.

A quick check with \textsf{GAP}~\cite{GAP} shows that almost simple primitive groups,
other than those in cases (a) and (b), fail the order bound
for $8$-et (and so for $k$-et for $8\le k\le n/2$) for degrees $n$ satisfying
$24\le n<2500$.

Let us call a primitive group $G$ of degree $n$ \emph{very small} if
$|G|\le n(n-1)(n-2)(n-3)$. Now a very small group satisfies the order bound
for $8$-et if $(n-4)(n-5)(n-6)<8!$, which holds only for $n\le39$. Very
small groups include all the rank~$1$ doubly transitive groups (those with
socle $\psl(2,q)$, $\mathrm{PSU}(3,q)$, $\Sz(q)$ and
$\mathrm{R}_1(q)$). Of these groups, further examination shows that only
$\pgaml(2,27)$ and $\pgaml(2,32)$ need further investigation.

A group with base size at most $6$ has order at most $n(n-1)\cdots(n-5)$.
So, if such a group satisfies the order bound for $k=8$, then
\begin{eqnarray*}
n(n-1)\cdots(n-5) &\ge& {n\choose 7}/8,\\
n-6 &\le& 8!=40320.
\end{eqnarray*}
This is beyond reasonable computational bounds, but we can do better. Note
first that, if $G$ satisfies $9$-et, then this result is improved to
$(n-6)(n-7)\le 9!$, or $n\le 609$. By ``upward closure'' of the order bound,
and the computer search, we can assume that $G$ has the $8$-et property
(unless its degree is at most $24$).

According to Lemma~\ref{l:clique_ind}, if $G$ has the $8$-et property,
then a non-trivial $G$-invariant graph contains no $7$-clique (and clearly
no $7$-coclique either). It follows from known bounds on Ramsey numbers (see
the survey \cite{rad}) that
$n<540$. Such a group (if almost simple of type (c)) is excluded by the
computer search mentioned earlier (except for groups of degree at most $24$).
Note that the weaker, and elementary, bound $R(7,7)\le{12\choose 6}=924$
would suffice here.

The remaining case consists of $2$-homogeneous groups. Such an almost simple
group is either ``very small'', or covered by case (b) above, or one of
finitely many others. Further inspection shows that the exceptions which
need to be considered are $M_{22}$ and its automorphism group, $M_{23}$,
$M_{24}$, and $\Co_3$. The Conway group can be excluded by \emph{ad hoc}
arguments. It fails the order bound for $k=9$ (and so for larger $k$), so
we may assume that $k=8$. It acts on a regular two-graph (a set of $3$-subsets)
which contains complete sub-hypergraphs of size $6$ and null sub-hypergraphs
of size $23$. Now an argument similar to that in Lemma~\ref{l:clique_ind}
gives a contradition to the $8$-et property.

We are left to check groups with degrees in the range $\{16,\ldots,24\}$ and
two larger examples with degrees $28$ and $33$. The last two are excluded since
they have too many orbits on $7$-sets ($29$ and $32$ respectively).

Let $G$ be primitive of degree $n$, where $16\le n\le 24$.
Filtering out those groups with more than $8$ orbits on $7$-sets, leaves just
nine groups (three of degree~$16$, two of degree $17$, and the Mathieu groups
including $\Aut(M_{22})$. Since the property is closed upwards, we only
need to consider $\agl(4,2)$, $\pgaml(2,16)$,
$\Aut(M_{22})$, $M_{23}$ and $M_{24}$. We outline arguments for these.

For $\mathrm{AGL}(4,2)$ we only need to consider $k=8$. There is an $8$-set
which is an affine subspace (containing no more than four independent points),
and a $6$-set with any $5$-subset independent; the second and a $7$-subset
of the first cannot coexist.

For $\pgaml(2,16)$, again we need to consider $k=8$.
A short computation shows that this group fails the weak $8$-et property.

For the Mathieu groups, we use Proposition~\ref{p:stab}, and an obvious
modification, to conclude that, if $M_{24}$ fails the weak $k$-et property,
then so do the other three groups. So consider first the case $G=M_{24}$, 
$8\le k\le 12$.

There is an $8$-set which is a block, and a $12$-set meeting no block in more
than $6$ points. The first and a subset of the second cannot coexist. This
excludes $k$ with $9\le k\le 12$.

Unfortunately $M_{24}$ does have the weak $8$-et property: it has just two
orbits on $7$-sets, and by connectedness there must be members of different
orbits meeting in six points. So we have to deal separately with the case $k=8$
for all the Mathieu groups.

For $M_{23}$, there is a $7$-set which is a block of the Steiner system, and
another meeting any block in at most $3$ points; these cannot coexist. A
similar argument applies to $\Aut(M_{22})$. Finally, for $M_{24}$, we resorted
to a computer search, as described earlier.
\qed\end{proof}

\section{The $k$-et property for $4\le k\le7$}

Let $4 \le k\le 7$, and $n\ge k/2$.
In this section, we will give a partial classification of all permutation groups $G$ on $n$ points that are $k$-et.  In some cases considered below, our arguments are repetitions of 
those used in the case that $k\ge 8$; we chose to give complete results to make this section self-contained. 

Essentially, the results of the previous sections reduce the classification problem to the case of $2$-homogeneous groups, potentially up to finitely many exceptions. All $2$-homogeneous groups are classified as a consequence of the CFSG and  the work of Kantor \cite{kantor:4homog, kantor:2homog}. The task is then to go through this list. At several points, we used GAP to check primitive groups for $k$-et, either by checking complete lists or dealing with large special cases. We give  a general outline of these checks.

To test whether a group has the $k$-et property, we check first whether it is
$k$-transitive (in which case the answer is yes), and then whether it satisfies
the order bound (if not, then the answer is no). If the case is not yet
decided, we make a list of orbit representatives on $k$-sets which are
witnesses for the weak $k$-et property (again, if none exists, then the
answer is no). For each such witness $B$, we build a $k$-partition of a subset
of $\Omega$, beginning with one in which each part is a singleton (these can
be constructed from orbit representatives on $k$-sets). Take a point not in 
this subset, and try adding it to each part of the partition, testing whether
the resulting partition has an image of $B$ as a section. If we reach a 
partition of $\Omega$ without this condition becoming true, we have found a
partition demonstrating that $B$ is not a witness; otherwise we conclude that $B$
is a witness. The program can thus find orbit representatives of all witnesses,
and certificates showing the failure of other $k$-sets.

We also note that the same program can be used to check the $k$-ut property;
simply check whether every orbit representative on $k$-sets witnesses $k$-et.

For our classification results, we will deal with $k=5,6,7$ together. Below, the group $\pxl(2,q)$, for $q$ an
odd square, denotes the extension of $\psl(2,q)$ by the product of diagonal
and field automorphisms of order~$2$.

\begin{theorem}\label{th567}
A permutation group $G$ of degree $n \ge 14$ satisfies $7$-et if and only if it satisfies one of the following:
\begin{enumerate}
\item $G$ fixes a point and acts $6$-homogeneously on the remaining ones;
\item $G=M_{24}$; 
\item $G$ is $7$-homogeneous.
\end{enumerate}
A permutation group $G$ of degree $n \ge 12$ satisfies $6$-et if and only if it satisfies one of the following:
\begin{enumerate}
\item $G$ fixes a point and acts $5$-homogeneously on the remaining ones;
\item $G$ is one of $\agl(4,2)$, $2^4:A_7$, $\pgl(2,17)$, $\pgaml(2,27)$, $\pgaml(2,32)$, $M_{11}(n=12)$, $M_{12}$, $M_{23}$, $M_{24}$;
\item $G$ is $6$-homogeneous. 
\end{enumerate}
A permutation groups $G$ of degree $n\ge10$ that satisfies one of the following properties has the $5$-et property:
\begin{enumerate}
\item $G$ fixes a point and acts $4$-homogeneously on the remaining ones;
\item $G=\pgaml(2,q)$ for prime powers $9\le q  \le 27$, or $q=32$, or $G$ is one of the subgroups $\pgl(2,9)$, $M_{10}$, $\psl(2,11)$, $\psl(2,16)$, $\psl(2,16):2$, $\pgl(2,25)$, $\pxl(2,25)$, $\pgl(2,27)$,
$\psl(2,32)$;
\item $G$ is one of $\psl(2,11)(n=11)$, $M_{11}(n=11,12)$, $M_{22}$, $M_{22}:2$, $M_{23}$; 
\item $G$ is $5$-homogeneous.
\end{enumerate}
The above list is complete, with the {\color{black}potential exception of $G= \pgaml(2,128)$.}
\end{theorem}

\begin{proof}
Let $k \in\{5,6,7\}$. It is clear that $k$-homogeneous groups are $k$-et, and the listed intranstive groups are $k$-et by proposition \ref{p:intrans}. The remaining sporadic groups listed in the theorem can be checked by computer to satisfy the listed $k$-et, with only the case of $\pgaml(2,32)$ and $k=6$ requiring extensive computation. 

Conversely, let $G$ be $k$-et. 
If $G$ is intransitive, then by Proposition \ref{p:intrans}, $G$ fixes one point and acts $(k-1)$-homogenously on the remaining points. If $G$ is transitive, then  by Proposition \ref{p:prim}, $G$ is primitive, in which case $G$ is either  $2$-homogeneous or $n\le R(k-1,k-1)$, by Proposition \ref{p:2homog}. 

Using GAP, we directly check all primitive groups of degree at most $32$, confirming the above results. In addition, we checked the primitive groups with degree up to the known upper limits on $R(k-1,k-1)$ against the order bound. 
The only non-$2$-homogeneous groups remaining were of the form $S_m$ acting on $2$-sets, or  $S_m \wr S_2$, as well some of their normal
 subgroups.  These can be ruled out as follows. Consider $S_m$ on pairs.
 The number of orbits of this group on $(k-1)$-sets is equal to the number of
 graphs (up to isomorphism) with $m$ vertices and $k-1$ edges. This number is easily seen to exceed $k$. 
For $S_m \wr S_2$, we can use the same argument, counting bipartite graphs with $2m$ vertices. 

Hence it remains to check the $2$-homogeneous groups of degree larger than $32$.  Those groups are either affine or almost simple. 

In the affine case all such groups are contained in $\agl(d,p)$, for some $d \ge 1$, and prime $p$. As in the previous section, if $\agl(d,p)$ does not satisfy $k$-et, then neither does any subgroup. 

Let $d\ge 3$. 
Choose  $l=\min(k-1, d)$ disjoint sets $P_i$  such that $\cup_{i=1}^j\, P_i$ is an affine subspace of dimension $j-1$, for $j=1,\dots, l$, and extend to a $k$-partition. This partition shows
that any potential witnessing set for $k$-et must span an affine subspace of dimension $l$. In contrast, consider a partition with $k-1$ singletons whose union lies in an affine subspace of dimension $h=\lceil \log_p (k-1)\rceil$. Any section of such a partition lies in a subspace of dimension at most $h+1$. These two requirements are incompatible for most values of $d,p,k$ with $p^d\ge 33$, showing that $\agl(d,p)$ is not $k$-et.

The remaining cases are as follows.
\begin{enumerate}
\item Several values with $p=2$. Here the result follows from Theorem \ref{t:ex.affine}.
\item $k=7$, $\agl(3,5)$, which fails the order bound.
\end{enumerate}
Now consider $\agl(2,p)$. As $p^2\ge 33$, we have $p \ge 7$, and so there exist $k-1$ points lying on an affine line. Moreover there  are sets of $4$ points for which every $3$-subset is affine independent. These  two sets cannot coexists in a witnessing set of size $k$, and so $\agl(2,p)$ is not $k$-et for $k\ge 5$.

Finally, let $d=1$. In this case, the order bound gives
$$p(p-1)\ge 2{p\choose k-1}/(k+1),$$ 
which fails for all relevant values of $p$ and $k$.

We next consider almost simple groups.
If $G$ has alternating socle,
 then $G$ is $k$-homogeneous and hence $k$-et. 

 Suppose next that $G$ has socle $\psl(2,q)$ for some $q=p^e$, $p$ prime, with its natural action.
 
 Let $k=5$. Orbits of $\pgl(2,q)$ on $4$-tuples of distinct elements are indexed by cross ratios of which there are $q-2$ values. The corresponding $4$-sets are indexed by sets of at most $6$ cross ratio values, hence $\pgl(2,q)$ has at least $(q-2)/6$ orbits on $4$-sets. In $\pgaml(2,q)$, field automorphisms can reduced this number by at most a factor of $1/e$. Hence for $n =q+1\ge 32$, $\pgaml(2,q)$ has too many orbits on $4$-sets to be $5$-et, unless potentially $q\in\{32, 49,64,81,128\}$. Additional computations exclude $q= 49, 64$, and
 confirm $q=32$, as well as  the subgroup $\psl(2,32)$.

We can exclude $\pgaml(2,81)$ by an argument based on circle geometries. This group preserves two type of circles with $4$ and $10$ elements, respectively. Choose circles 
$C \subset C'$ of different type, and consider a $5$-partition of the projective plane into sets $P_i$ such that $P_1\cup P_2 \cup P_3 =C$, $P_4=C' \setminus C$. Any section of such a partition cannot contain a circle of the smaller type. However, we can create a partition whose sections contain such a circle by using $4$ singletons sets. This leaves the case of $\pgaml(2,128)$ (its socle $\psl(2,128)$ can be excluded by the orbit counting argument from above).
 
 If $k=6, 7$, the only group of degree at least 33 that does not fail the order bound is $\pgaml(2,32)$, for $k=6$. This group was confirmed to be $6$-et by an extensive computation.
 
 Consider next the case of $G = \pgaml(3,q)$ with its action on projective points. These groups do not satisfy $k$-et, as we may find a set of $k-1$ 
 projective points that lie within a hyperplane, and a set of $4$ points in which all $3$-subsets span the projective plane. These two sets cannot coexist in a witnessing set.  A similar argument excludes $\pgaml(d,q)$ with $d\ge 4$: in most cases, we may choose a set of $k-1$ points that lies in a flat of minimal possible rank, and a set of size $l=\min(k-1, d)$ which spans a flat of rank $l-1$. For a few cases with $d=4$, we also require that every space spanned by a $(k-1)$-subset of the latter set has maximal possible rank. 
 
For  $G$ with socle $\mbox{PSU}(3,q)$, ${}^2B_2(q)$, or ${}^2G_2(q)$, $n\ge 33$, the order bound  fails except for $\pgamu(3,4)$, $k=5$. This case can be excluded by having too many orbits on $4$-sets.
 
Consider  $\Sp(2d,2)$ in either $2$-transitive representation. For $k=5$ note that a full and empty $4$-set (in the notation of Theorem \ref{t:ex.Sp}) cannot coexists in a $5$-set. For $k=6,7$ we may instead replace the full $4$-set with one of size $k-1$ in which each $3$-subset is an element of 
 the designated orbit $T$.  As demonstrated in \cite[Section 2.6]{ArCa}, these sets exist up to size $2^{d-1}$ in the $-$ case and size $2^d$ in the $+$ case, which is sufficient to cover all cases with $n \ge 32$. Hence $\Sp(2d,2)$ is not $k$-et.
 
 The remaining sporadic cases all have $n\le 32$, except for the Conway group $\Co_3$  and the Higman-Sims group. $\Co_3$ is not $k$-et for $5\le k \le 7$ on account of having too many orbits on $(k-1)$-sets. 
 
Finally,  HS fails the bound for $k\ge 6$, and has too many orbits on $4$-sets to be $5$-et.\qed
\end{proof}

To obtain a classification for $4$-et, we first establish a results about the action of $\pgu(3,q)$. This group acts on a $3$-dimensional vector space $V$
over $\gf(q^2)$, preserving a nondegenerate Hermitian form $H$ (a sesquilinear
form with zero radical satisfying $H(w,v)=H(v,w)^q$). It acts $2$-transitively
on the \emph{unital} $U(q)$, the set of $1$-dimensional subspaces of $V$
on which $H$ vanishes; any two points of the unital lie on a unique line of the
projective space, meeting the unital in $q+1$ points (so these lines are the
blocks of a Steiner system $S(2,q+1,q^3+1)$).

\begin{prop}
The number of orbits of the group $\pgu(3,q)$ on $3$-element subsets of $U(q)$
is $(q+3)/2$ if $q$ is odd, $(q+2)/2$ if $q$ is even. Apart from one orbit
consisting of collinear triples, these orbits are parametrised by inverse
pairs of elements of $\gf(q^2)^\times/\gf(q)^\times$ excluding the coset
$\{x\in\gf(q^2)^\times:x^q=-x\}$.
\label{p:umain}
\end{prop}

\paragraph{Remark} The parametrisation allows us to count orbits of $G$ with 
$\pgu(3,q) \le G \le \pgamu(3,q)$ on $3$-sets; these just correspond to orbits of the corresponding subgroup of the Galois
group on the pairs of cosets described.

\medskip

We begin with a preliminary result.

\begin{lemma}
Let $(v_1,v_2,v_3)$ and $(w_1,w_2,w_3)$ be two bases for $V$, and let
$P_i=\langle v_i\rangle$ and $Q_i=\langle w_i\rangle$ for $i=1,2,3$. Let
$a_{ij}=H(v_i,v_j)$ and $b_{ij}=H(w_i,w_j)$. Then
\begin{itemize}\itemsep0pt
\item[(a)] The element $c=a_{12}a_{23}a_{31}$ satisfies $c^q+c\ne0$;
\item[(b)] $(P_1,P_2,P_3)$ and $(Q_1,Q_2,Q_3)$ lie in the same orbit of
$\pgu(3,q)$ if and only if $a_{12}a_{23}a_{31}$ and $b_{12}b_{23}b_{31}$ lie
in the same coset of $\gf(q)^\times$ in $\gf(q^2)^\times$.
\end{itemize}
\end{lemma}

\paragraph{Proof} (a) The Gram matrix of $\{v_1,v_2,v_3\}$ relative to the
form is
\[\pmatrix{0&a_{12}&a_{13}\cr a_{21}&0&a_{23}\cr a_{31}&a_{32}&0\cr}.\]
Since $H$ is nondegenerate, this matrix must be nonsingular. But its
determinant is
\[a_{12}a_{23}a_{31}+a_{13}a_{32}a_{21}=c+c^q,\]
since $a_{21}=a_{12}^q$ and so on.

\smallskip

(b) By Witt's theorem \cite[p.57]{taylor}, there is an element of $\pgu(3,q)$
mapping $(v_1,v_2,v_3)$ to $(w_1,w_2,w_3)$ if and only if
$H(v_i,v_j)=H(w_i,w_j)$ for all $i,j$. In order to map the points spanned by
the first three vectors to those spanned by the second, we have to map
$(v_1,v_2,v_3)$ to $(x_1w_1,x_2w_2,x_3w_3)$ for some scalars $(x_1,x_2,x_3)$.
This requires $a_{ij}=x_ix_j^qb_{ij}$, and so
\[a_{12}a_{23}a_{31}=(x_1x_2x_3)^{q+1}b_{12}b_{23}b_{31};\]
so $a_{12}a_{23}a_{31}$ and $b_{12}b_{23}b_{31}$ differ by a $(q+1)$st power
factor (i.e. an element of $\gf(q)^\times$).

Conversely, if this is the case, we can adjust the vectors by scalar factors
to ensure that $a_{12}=a_{31}=b_{12}=b_{31}=1$ and $a_{23}=b_{23}$, so the
two triples lie in the same orbit. The adjustments introduce $(q+1)$st power
factors into the expressions $a_{12}a_{23}a_{31}$ and $b_{12}b_{23}b_{31}$.

\paragraph{Proof of Proposition~\ref{p:umain}} We know that two triples of
points lie in the same orbit if and only if the expressions $a_{12}a_{23}a_{31}$
lie in the same coset of $\gf(q)^\times$, and that one coset is excluded. So
there are $q$ orbits on such (ordered) triples.

It follows from the lemma that each triple is invariant under a subgroup of
$\pgu(3,q)$ which permutes its elements cyclically, so we only have to decide
whether there is an element of this group which induces a transposition on
such a triple. For this to hold, $a_{12}a_{23}a_{31}$ and $a_{13}a_{32}a_{21}$
must lie in the same coset of $\gf(q)^\times$. These elements are $c$ and
$c^q$; so the map $x\mapsto x^q$ must fix this coset. This means that
$c^{q-1}\in\gf(q)^\times$, so $c^{(q-1)^2}=1$. It follows that $c^{2(q-1)}=1$,
so that $c^{q-1}=1$ or $c^{q-1}=-1$. The second possibility was excluded by
part (a) of the Lemma. If $q$ is odd, there remains just one such coset; if
$q$ is even, the two cases are the same. The other cosets ($q-1$ or $q$
depending on the parity of $q$) are permuted in $2$-cycles by this
transformation. So the number of orbits is $1+(q-1)/2=(q+1)/2$ if $q$ is odd,
and $q/2$ if $q$ is even.

Adding one (for the single orbit consisting of collinear triples) gives the
result of the Proposition.\qed

\begin{theorem}\label{th:4-et}
Let $G$ be a permutation group of degree $n\ge 8$. If $G$ satisfies any of the following conditions, then $G$ has the $4$-et property.
\begin{enumerate}
\item $G$ fixes a point and acts $3$-homogeneously on the remaining ones;
\item $G$ is one of $\HS$ or $\HS:2$ with their action on $100$ points\label{l:HS};
\item $G=\agl(d,2)$, $d \ge 3$;
\item $G$ is a $2$-transitive subgroup of $\agl(3,2)$ or $\agl(2,3)$;
\item $G$ is one of  $\agl(1,16):2$, $\agaml(1,16)$, 
$\asl(2,4)$, $\asl(2,4):2$, $\agl(2,4)$, $\agaml(2,4)$, $2^4.A_6$, $2^4.A_7$, $\asl(2,5)$, $\asl(2,5):2$, $\agl(2,$ $5)$,
$\agaml(2,8)$, $\asigl(2,8)$, $\agl(2,8)$, $\agl(1,11)$, $\agl(1,13)$, $\agaml(1$, $32)$, $\agaml(2,32)$,  $2^6: G_2(2)$, or $2^6: \psu(3,3)$;
\item $G=2^d: \Sp(d,2)$, $d \ge4$ and even;
\item $\psl(2,q)\le G \le \pgaml(2,q)$ for prime powers $q$ with $7 \le q \le 49$;
\item $\psl(3,q) \le G \le \pgaml(3,q)$, for prime powers $q\ge 3$;
\item $\psu(3,q) \le G \le \pgamu(3,q)$, for $q\in\{3,4\}$;
\item $G=\Sp(2d,2)$, $d \ge 3$, in either of its $2$-transitive representations;
\item $G$ is one of $\psl(2,11)(n=11)$, $M_{11}(n=12)$, $M_{22}$, $M_{22}:2$, $\Sz(8).3$, $\pgaml(2,128)$, $Co_3$; 
\item $G$ is $4$-homogeneous.
\end{enumerate}
{\color{black}If any other groups $G$ are  $4$-et, then they satisfy one of the following:
\begin{enumerate}
\item $\psl(2,q) \le G\le \pgaml(2,q)$ for some prime power $q \ge 51$, $G \ne \pgaml(2,128)$; 
\item $G \in\{\pgu(3,5)$, $\pgamu(3,5)$, $\psu(3,8).3$, $\psu(3,8).6$, $ \psu(3,8).3^2$, $\pgamu(3,8)$, $\pgamu(3,9)$,
$\pgamu(3,16)$, $\Sz(8)$,  $\Sz(32):5$, $\HS$ $(n=176) \}$. 
\end{enumerate}
In the last case, note that there are $3$ non-isomorphic groups of the form $\psu(3,8).3$. Only one of those has less than $5$ orbits on $3$-sets and could be $4$-et.}
\end{theorem}

\begin{proof}
If $G$ is intransitive, the result follows from Proposition \ref{p:2homog}. Transitive, but imprimitive groups are excluded by  Proposition \ref{p:prim} with possible exceptions for groups with $n \in \{8,9\}$. These cases can be handled exactly as in the proposition, as the premise $n>9$ was only needed in a different subcase. 
If $G$ is primitive, but not $2$-homogeneous, then it is $4$-et exactly if listed under (\ref{l:HS}), by Theorem \ref{t:4et2hom}.

Hence it remains to classify the $2$-homogeneous groups satisfying $4$-et. For groups of degree at most $50$, we can do so directly using GAP, confirming the above results. 
Assume that $n\ge 51$. 

If $G$ is $2$-homogeneous, but not $2$-transitive, then by \cite{kantor:2homog}, $G$ is contained in a one-dimensional affine group, while a $2$-transitive group is either affine or almost simple.
We will address the affine cases first.

Any such group is contained in $\agl(d,p)$ for some $d\ge 1$ and prime $p$. As above, if $\agl(d,p)$ does not satisfy $4$-et, then neither does any subgroup. If $p=2$, then 
$\agl(d,2)$ has the $4$-et property by Theorem \ref{t:ex.affine}. 

So, let $p\ge 3$, and consider first  $d\ge 3$. 
Choose three disjoint sets $P_i$  such that $\cup_{i=1}^j\, P_i$ is an affine subspace of dimension $j-1$, for $j=1,\dots, 3$, and extend to a $4$-partition. This partition shows
that any potential witnessing set for $4$-et must span an affine subspace of dimension $3$. Using $3$ singletons contained in an affine line, we can construct another partition whose sections are contained in an affine space of dimension at most $2$, showing that $\agl(d,p)$ is not $k$-et.

If $d=2$, then $p \ge 11$, as $n\ge 51$.  By adopting the partition from the case $d\ge 3$, we see that any potential witnessing set for $4$-et must contain $3$ points on an affine line, and one point not on the line. It follows that $G$  needs to act transitively on $3$-sets of collinear points.
However, for a given collinear triple $(x_1,x_2,x_3)$ of distinct points, $\agl(2,p)$ preserves the value $\lambda \in F_q\setminus\{0,1\}$ satisfying $x_2-x_1=\lambda (x_3-x_1)$. By permuting the $x_i$, at  most $6$ different values of $\lambda$ arise. Hence for $p\ge 11$, there are at least $2$ orbits of collinear $3$-sets.
It follows that $\agl(2,p)$ is not $4$-et for $p\ge 11$.

Finally, let $d=1$. In this case, the order bound fails for $p=n\ge 51$.

It remains to examine the $2$-homogeneous subgroups of $\agl(d,2)$ for $d\ge 6$.

Consider the groups $\agaml(e,2^{d/e})$, for $e$ properly dividing $d$. These can be handled similarly to $\agl(d,p)$, $p\ge 3$, except that field automorphisms change some of the numerical estimates involved.  Concretely, if $e\ge 3$, the same argumentation shows that $\agaml(e,2^{d/e})$ is not $4$-et. If $e=1$, the order bound shows that $\agaml(1,2^l)$ is not $4$-et for $l \ge 7$. This leaves the case $\agaml(1,64)$, which can be excluded be special computation.
If $e=2$, then each value of $\lambda$ in the  argument of the prime case may be mapped to an additional $d/2$ values due to field automorphisms. The argument now carries through to show that $4$-et fails for $2^{d/2} > 32$, leaving the cases $\agaml(2,8)$, $ \agaml(2,16)$, and $\agaml(2, 32)$. Computation shows that $\agaml(2,8)$ and its listed subgroups are $4$-et, while $\agaml(2,16)$ can be embedded into 
$\agaml(4,4)$ and thus is not $4$-et. 

For $G=\agaml(2,32)$, we can use a similar argument as in Theorem \ref{t:ex.psl}. The $4$-sets in which exactly $3$ elements lie on an affine line form an orbit $O$ of $G$, as the stablizer of a line acts $3$-transitively. Consider a $4$-partition $\mathcal{P}=(P_1,P_2,P_3,P_4)$ with $|P_1|\le |P_2|\le |P_3|\le |P_4|$. We claim that we can choose a line meeting at least three parts of $\mathcal{P}$. 
Choose $x \in P_1, y \in P_2, z \in P_3$. If the line $L$ through $x$ and $y$ intersect $P_3 \cup P_4$, we can choose $L$. Otherwise $L \subseteq P_1 \cup P_2$. Now $P_4$ 
has at least $256$ elements, at most $31$ of which lie on the line through $z$ and parallel to $L$. Hence we may chose $w\in P_4$ not on this line, in which case the line through 
$z$ and $w$ intersects $L$ and hence one of $P_1$ or $P_2$. We can now see that $O$ contains a section of $\mathcal{P}$ by the same argument as in Theorem \ref{t:ex.psl}. 

Considering the subgroups of $\agaml(2,32)$, note that $\agl(2,32)$ is not $4$-et, as it has more than one orbit on $3$-sets of collinear points. For triples 
$\vec x=(x_1,x_2,x_3)\in F_{32}^2$, let $A_{\vec x}$ be the matrix with columns $x_1+x_2, x_1+x_3$. Now $\det A_{\vec x}$ is invariant under 
the  induced action of $\asl(2,32)$ as well as under permutation of the arguments. It follows that $\asl(2,32)$ has at least $32$ orbits on $3$-sets, and hence $\asigl(2,32)$, has 
at least $7$ and is not $4$-et. 

It remains to check subgroups of $\agl(d,2)$ that are not contained in any $\agaml(e,2^{d/e})$ for $d\ge 6$, $e\ge 2$. The groups $2^d:\Sp(d,2)$, with $d$ even, were shown to be 
$4$-et in Theorem \ref{t:ex.Sp}. Finally, for $d=6$, there are two more sporadic cases ($2^6:G_2(2)$ and its subgroup  $2^6: \mbox{PSU}(3,3)$), which can be handled computationally.

We next cover the case that $G$ is a $2$-transitive almost simple group of degree $n$.  We may assume that $G$ is not $4$-homogeneous.

Let $G$ have socle $\psl(d,q)$ for $d\ge 2$, $q=p^e$, $p$ prime. For $d=2$, these cases are currently open above the computational range, with $\pgaml(2,128)$ confirmed to have 
$4$-et by computation as well. 
By Theorem \ref{t:ex.psl}, groups $G$ with $\psl(3,q) \le G \le \pgaml(3,q)$ are $4$-et. If $d\ge 4$,  we can exclude $\pgaml(d,q)$ by a now familiar argument: choose a partition containing $3$ points on a projective line, and another that forces every section to span a projective $3$-space. 
 
  If $G$ has unitary socle, we can calculate the number of orbits on $3$-sets by Proposition \ref{p:umain} and the remark following it. This count excludes all values of $q$ except $4,5,8,9,16$. Proper subgroups of $\pgamu(2,16)$ can be excluded by this argument as well. We can directly compute the number of orbits for proper subgroups of $\pgamu(3, q)$ for $q=5,8,9$, which excludes all groups not listed in the theorem. Finally
 $\psu(3,4)$ was confirmed to be $4$-et by direct computation.

 If $G$ has socle $\Sz(q)$, or ${}^2G_2(q)$,
  then eventually the order bound will fail. For $n\ge 51$, this leaves 
 only 
 $\Sz(q)$, $q\in\{8,32\}$. Computation confirms that $\Sz(8).3$ has $4$-et, and that $\Sz(32)$ has $6$ orbits on $3$-sets, and hence is not $4$-et.

 $\Sp(2d,2)$ (in either $2$-transitive representation) was shown to be $4$-et in Theorem \ref{t:ex.Sp}.  
 The remaining sporadic socles all have degree less than $51$, except for the Conway group $\Co_3$  and the Higman-Sims group. $\Co_3$ was shown to be $4$-et in Theorem \ref{t:ex.Sp}. For $HS$, $4$-et is open.
\qed
 \end{proof}

\section{The $k$-ut condition}
\label{s:ut}

In this section we extend the work on the $k$-ut condition in \cite{ArCa} by
classifying some of the previously unresolved cases. We would also like to
record here the correction of a couple of small mistakes in \cite{ArCa}.

In the proof of Proposition 2.6 of that paper, the authors
assert ``a short calculation yields $n\le k+2$'': this is not correct, but it is easy to fix.

The situation is that we have a vertex-primitive graph $\Gamma$ whose valency
$v-1$ is smaller than $k$, such that every $(k+1)$-set contains a closed
vertex-neighbourhood in the graph, and wish to reach a contradiction.
Now by a theorem of Little, Grant and
Holton~\cite{lgh}, $\Gamma$ has a near $1$-factor (a collection of pairwise
disjoint edges covering all or all but one of the vertices). If $n\ge 2(k+1)$,
a $(k+1)$-set containing at most one vertex from each edge of the near
$1$-factor yields a contradiction. In the remaining case $n=2k+1$, let $w$ be
the uncovered vertex. If two vertices $x,y$ in the neighbourhood of $w$ form an
edge of the partial $1$-factor, take $x$ and $y$ and one point from each
remaining edge; if not, take $w$ together with one point from each edge of the
partial $1$-factor (if the first choice contains the closed neighbourhood of
$w$, replace one vertex by the other end of the edge in the partial $1$-factor).

\medskip

In addition, two specific groups were omitted from the list of groups with the
$3$-universal transversal property \cite[Theorem 4.2(4)]{ArCa}, namely
$\psl(2,11)$ (degree~$11$) and $2^4:A_6$ (degree~$16$). It is easy to verify
these directly; but they are both handled by a general result which also has
applications to the $3$-existential transversal property, which we give here.
(Both these two specific groups satisfy the conditions of the last sentence
of the Proposition following; this also deals with cases (i), (ii), (iii) and
(v) of \cite[Theorem 4.2(4)]{ArCa}.

\begin{prop}
Let $G$ be a $2$-primitive permutation group of degree~$n$, and let $\Delta$
an orbit of the stabiliser of two points $x,y$ which has cardinality greater
than $n/3-1$. Then the set $\{x,y,z\}$, for $z\in\Delta$, witnesses the $3$-et
property. In particular, if all $G_{xy}$ orbits have size greater than $n/3-1$,
then $G$ has the $3$-ut property.
\label{p:watkins}
\end{prop}

\begin{proof}
The images of $\{y,z\}$ under $G_x$ form an orbital graph for this group,
with valency $k=|\Delta|$ (or possibly twice this number, if $\Delta$ is a
non-self-paired suborbit of $G_x$). This graph is vertex-primitive, so by a
theorem of Watkins~\cite{watkins}, its vertex-connectivity is at least $k$.
(Although Watkins does not state this explicitly, it is a simple consequence
of his results: in his terminology, atomic parts are blocks of imprimitivity;
if the vertex-connectivity is less than the valency then these blocks are
non-trivial.)

Take any $3$-partition $\mathcal{P}$ of $\Omega$, with smallest part $A$ of size $l$,
where $l\le n/3$. Without loss of generality, $x\in A$. Now by hypothesis,
$l-1<k$; so removing $l-1$ points from the graph $\Gamma$ leaves a connected
graph. This graph has an edge $\{u,v\}$ which is a transversal to the
$2$-partition of $\Omega\setminus A$ formed by the other two parts of $\mathcal{P}$; thus
$\{x,u,v\}$ is a transversal to $\mathcal{P}$ and is an image of $\{x,y,z\}$, as
required. \qed
\end{proof}

We now extend the results of \cite{ArCa} by addressing some cases left open. Our first technical result also has some relevance with regard to the $k$-et question. Recall that the orbits of $\pgl(2,q)$ on ordered distinct $4$-tuples are indexed by cross ratios from $F_q\setminus\{0,1\}$. The corresponding orbits on $4$-sets are then given by sets of usually six, but occasionally fewer, cross ratio values. If $\pgl(2,q) \le G\le \pgaml(2,q)$, it follows that the $G$-orbits on $4$-sets are also indexed by sets of cross ratios.

\begin{lemma}
Let $\pgl(2,q) \le G\le \pgaml(2,q)$, and $O$ an orbit of $G$ on $4$-sets. If the cross ratios associated with $O$ do not generate the multiplicative group $F_q^*$, then an element of $O$ does not witness the $4$-et property for $G$.
\label{l:crossratio}
\end{lemma} 

\begin{proof} Let $M$ be the subgroup of $F_q^*$ generated by the cross ratios associated with $O$. Partition the projective line into $\{ \infty \}, \{0\}, M, F_q^*\setminus M$, and consider any section $(\infty, 0, x,y)$. One of the possible orders results in a cross ratio of $x/y$. This element cannot be in $M$, and hence cannot be one of the cross ratios indexing $M$. It follows that the elements of $O$ do not witness $4$-et. \qed
\end{proof}

\begin{cor}
Suppose that $p\ge 13$ is prime, and $p \not\equiv 11 \mbox{ mod } 12$. Then $\pgl(2,p)$ (and its subgroups) do not satisfy $4$-ut. 
\end{cor}

\begin{proof} If $p \equiv 1 \mbox{ mod } 3$, then $F_q$ contains a primitive sixth root of unity $\omega$. An orbit with this cross ratio has the property that other cross ratios lie in the group of sixth roots of unity.
The result for $\pgl(2,p)$ now follows directly from the lemma.

If $c$ is one of the cross ratios of an orbit, the corresponding subgroup $M$ of $F_q^*$ is generated by $c,c-1, -1$. If $p \equiv 1 \mbox{ mod } 4$, then, as detailed in the remarks after  \cite[Theorem 5.3]{ArCa}, there are values $c, c-1$ that are both squares in $F_q$. In addition $-1$ is a square and so these values generate a subgroup of the group of squares of $F_q^*$. The result now follows again from the lemma.  \qed
\end{proof}

In addition to the results above, we have settled several remaining open cases computationally. The groups $\pgl(2,7)$, $\pgaml(2,128)$, and $\psl(2,q)\le G\le \pgaml(2,q)$ for $8,11,23, 32, 47$ satisfy $4$-ut, while $\psl(2,7)$
does not. Finally, $\Sz(8)$ and $\Sz(8):3$ satisfy $3$-ut. 

On the basis of our computations, we venture the conjecture that the converse
of Lemma~\ref{l:crossratio} is also true.

\section{Applications to semigroups}\label{app}

A semigroup $S$ is said to be (von Neumann) regular if for every $x\in S$ there exists $x'\in S$ such that $x=xx'x$. Some of the most important classes of semigroups (such as groups, inverse semigroups, completely regular semigroups, the endomorphism monoid of a vector space or of a set, etc.) are contained in the class of regular semigroups, and the theory is rich enough to allow some of the deepest and most interesting results in semigroups. 

Regarding the general aim of using the powerful tools in group theory to extract information about semigroups (studying the interplay between the structure of a semigroup and its group of units), the ultimate goal is to classify the pairs $(G,t)$, where $G\le S_n$ is a group of permutations of some $n$-set and $t$ is a transformation of the same set, such that $\langle G,t\rangle$ has a prescribed  property $P$. This problem, in its full generality, was solved for a particular instance of $P$ in \cite{AAC}. Given the current state of our knowledge, a full solution of this problem is totally hopeless when $P$ is the property of being a regular semigroup. Nevertheless, in previous investigations, it was possible to solve particular, yet very interesting, instances of this general problem.  For example, we have the classification of  the groups $G\le S_n$ such that $\langle G,t\rangle$ is regular, for all  $t\in T_n$ \cite{ArMiSc}; then, resorting on a much deeper analysis, we found the classification of the groups $G$ that together with any rank $k$ map (for a fixed $k\le n/2$)  generate a regular semigroup \cite{ArCa}. Now our goal is to move a step forward classifying the groups $G\le S_n$ such that $\langle G,t\rangle$ is regular, for all maps $t$ with image a given set. 

Let $n$ be a natural number and let $X:=\{1,\ldots,n\}$ be a set. Let $k\le n/2$ and let $B\subseteq X$ be a $k$-set. Denote by $T_{n,k}$ the set of rank $k$ maps in $T_n$; denote by $T_{n,B}$ the set of maps in $T_n$ whose image is $B$. Of course $T_{n,B}\subseteq T_{n,k}$. 

As said above, we have the classification of the groups $G\le S_n$ such that  $\langle G,t\rangle$ is regular, for all $t\in T_{n,k}$; the goal now is to tackle the much more ambitious problem of classifying the groups $G\le S_n$ such that $\langle G,t\rangle$ is regular, for all $t\in T_{n,B}$ with $B$ being a given $k$-set. 

The next result provides a necessary condition these latter groups must satisfy. 
\begin{theorem}\cite[Theorem 2.3 and Corollary 2.4]{lmm}\label{aux1}
Let $G\le S_n$ and $t\in T_n$. 
Then  the following are equivalent: 
\begin{itemize}
\item $t$ is regular in $\langle G,t\rangle$;
\item there exists $g\in G$ such that $\rank(t)=\rank(tgt)$;
\item the elements in $\langle G,t\rangle$ having the same rank as $t$ are regular. 
\end{itemize}
\end{theorem}

Let $B$ be a finite set contained in $\{1,\ldots,n\}$. It follows from this result that if $B$ witnesses $|B|$-et, then any $t\in T_{n,B}$ is regular in $\langle G, t \rangle$. In fact, 
if $t', s\in \langle G,t \rangle$,  $\rank(t')=\rank (t's)$, and the image $I$ of $t's$ witnesses $|I|$-et, then $t'$ is regular in  $\langle G, t \rangle$.
Converesely, if $\langle G,t\rangle$ is regular for all $t\in T_{n,B}$, then $G$ has the $|B|$-et property and $B$ witnesses it. 
This observation together with Theorem \ref{egregious} immediately implies the following. 

\begin{cor}
Let $X=\{1,\ldots,n\}$, let $8\le k\le n/2$ and let $B\subseteq X$ be a $k$-set. 
Let $G\le S_n$ be transitive. If $\langle G,t\rangle$ is regular for all $t\in T_{n,B}$, then $G$ is $A_n$ or $S_n$. Conversely, if $G$ is $A_n$ or $S_n$, then for any $k$-set $B$ 
and $t\in T_{n,B}$, $\langle G,t\rangle$ is regular.
\end{cor}

In order to handle the intransitive case and the remaining values of $k$, we need some more considerations. Fix $k$ such that  $k\le n/2$. Let $G$ be a group possessing the $k$-et property, and suppose $B$ witnesses it. This means that any map $t\in T_{n,B}$ is regular in $\langle G,t\rangle$; in fact, by Theorem \ref{aux1} we know that every map of the same rank as $t$ is regular. 
Therefore,  in the semigroup $\langle G,t\rangle$ we have:
\begin{enumerate}	
\item the elements of $G$, which are all regular; 
\item the elements with rank $k$, which are all regular; 
\item the elements whose rank is less than $k$.
\end{enumerate}
The conclusion is that the semigroup $\langle G,t\rangle$ will be regular if the lower rank maps are regular. As constants are idempotents (hence regular) it follows that for $k=2$, the semigroup will be regular. 
 
Regarding larger values of $k$, the easy way of ensuring regularity of the semigroup is to require the group to have the $(k-1)$-ut.  These observations are summarised in the following theorem. 

\begin{theorem}	\label{semimain}
Let $X=\{1,\ldots,n\}$, let $2\le k\le n/2$ and let $B\subseteq X$ be a $k$-set. Let $G\le S_n$  be a group possessing the $k$-et property (witnessed by $B$)
and, in addition, possessing the $(k-1)$-ut property. Then
$\langle G,t\rangle$ is regular, for all $t\in T_{n,B}$. 
\end{theorem}

\begin{proof}
As seen above, the elements in $G$ and the rank $k$ elements are regular. The fact that the group has the $(k-1)$-ut property guarantees that the rank $k-1$ elements are also regular. In addition, by \cite{ArCa}, we know that a group with the $(k-1)$-ut property possesses the $(k-2)$-ut property. The result follows by repeated application of the foregoing argument. 
\qed\end{proof}
By essentially the same argument, if $G$ satisfies $k$-et and $l$-ut for some $l <k$, it suffices to show that all elements with rank strictly between $k$ and $l$ are regular to establish regularity of $\langle G, t\rangle$. In fact, except for the intransitive groups and $2$ further examples, all relevant groups satisfy $k$-et and $(k-2)$-ut, reducing the problem to examining elements of rank $k-1$. The following lemma address these elements. Its additional assumption also hold in nearly all cases.
\begin{lemma}\label{l:reg-1} Let $G\le S_n$ be $k$-et with witnessing set $B$, as well as $(k-1)$-et, but not $(k-1)$-ut.  Let $\bar B \subset B$ be a subset that does 
not witness $(k-1)$-et, such that no other $(k-1)$-subset of $B$ belongs to the orbit of $\bar B$.  
\begin{enumerate}
\item \label{e:notreg} Let $t \in T_{n,B}$, and $\mathcal{P}=\{P_1,\dots, P_k\}$ be the kernel of $t$, such that 
$P_kt = B \setminus \bar B$. Suppose that every $(k-1)$-subsection of $\mathcal{P}$ containing an element of $P_k$ does not lie in the orbit of $\bar B$, and that the following holds for some $g \in G$:
\begin{enumerate}
\item $Bg$ omits exactly the kernel class $P_k$.
\item The (unique) two elements of $Bg$ lying in the same class of $\mathcal{P}$ are in $\bar Bg$.
\end{enumerate}
Then $tgt$ has rank $k-1$ and is not regular in $\langle G,t\rangle$.
\item \label{e:reg} Assume that the orbit of $\bar B$ is the only one not witnessing $(k-1)$-et. Suppose that every $k$-partition $\mathcal{P}$ of $\Omega$ satisfies the following condition: 

If there exists a part $P$, such that every $(k-1)$-subsection of $\mathcal{P}$ intersecting $P$ witnesses $(k-1)$-et, then every  $(k-1)$-subsection of $\mathcal{P}$ not intersecting $P$ does not witnesses $(k-1)$-et. 

Then 
for every $t \in T_{n,B}$, all rank $k-1$ elements in $\langle G,t \rangle$ are regular.
\end{enumerate}
\end{lemma}
\begin{proof} Assume first that the conditions of (\ref{e:notreg}) hold for $g \in G$, and consider $tgt$.

Note that $tgt$ has image $\bar B$ and rank $k-1$. We claim that it is not regular in $\langle G, t\rangle$. For let $I$ be in the orbit of $\bar B$. Our conditions on $\mathcal{P}$ imply that either $It=\bar B$, and hence in the same orbit,  or has rank less than $k-1$. By induction,
the image of $tgts$ does satisfy one of these two conditions for all $s\in \langle G,t\rangle$. Now, the kernel if $tgt$ is obtained from $\mathcal{P}$ by merging $P_i$ and $P_j$ with $i,j < k$. It follows that every section of the kernel of $tgt$ contains an element of $P_k$ and hence is not in the orbit of $\bar B$. Thus the image of $tgts$ is not a section of the kernel of $tgt$, and so $tgtstgt$
has rank less than $k-1$, for all $s\in \langle G,t\rangle$. Thus $tgt$ is not regular.

Now assume that the conditions of (\ref{e:reg}) hold. Let $t\in T_{n,B}$ and $t' \in   \langle G,t \rangle$ have rank $k-1$. Let $\mathcal{P}$ be the kernel of $t$. 

If the image $I$ of $t'$ witnesses $(k-1)$-et, then $t'$ is regular. 
So assume that $I$ does not witness $(k-1)$-et. Suppose there is a $(k-1)$-subsection $S$ of $\mathcal{P}$ that does not witness $(k-1)$-et, and that $St \ne \bar B$. As only one orbit of $G$ does not witness $(k-1)$-et, there exists $g \in G$ mapping $I$ to $S$. Then $t'gt$ has image witnessing $(k-1)$-et, and $t'$ is regular.

Otherwise, there is a kernel class $P$, the preimage of $B \setminus \bar B$, such that every $(k-1)$-subsection of $\mathcal{P}$ intersecting $P$ witnesses $(k-1)$-et. Hence every $(k-1)$-subsection of $\mathcal{P}$ not intersecting $P$ does not witnesses $(k-1)$-et. Write $t'$ as a product of the generators in $G \cup\{t\}$. 
In this product, consider the first 
occurrence of a subterm $tg't$ such that $tg't$ has rank $k-1$. If the image $I'$ of $tg't$ witnesses $(k-1)$-et, then $I't \ne \bar B$, and hence also witnesses $(k-1)$-et. Clearly, 
so does $I'g$ for any $g\in G$, which implies that $I$ witnesses $(k-1)$-et, contrary to assumption. Hence $I'=\bar B$.  However, this implies that the image $Bg'$ of $tg'$ intersects all kernel classes other than $P$. Thus $Bg'$ is the union of two sections of $\mathcal{P}\setminus \{P\}$. One of these sections is not  $\bar Bg'$, and hence witnesses $(k-1)$-et,
contrary to our assumption. Hence this case cannot occur, and $t'$ is regular. \qed

\end{proof}
We have automated part of the search process required by the second part of the lemma. Starting with a partial partition containing the elements of $B$ as singletons, only the part $B \setminus \bar B$ can play the role of $P$. We extend the partial partition by single elements to obtain all partitions in which every $(k-1)$-subsection intersecting $P$ witnesses $(k-1)$-et, pruning partial partition that already violate this condition. All such partition can then be checked to see if they satisfy the additional conditions of Lemma \ref{l:reg-1}.

We will now address the individual groups satisfying $k$-et, starting with the intransitive case. 
\begin{prop} \label{p:regintrans} Let $2 < k\le n/2$, and $G\le S_n$ be an intransitive group, that satisfies $k$-et with witnessing set $B$. Then $\langle G,t\rangle$ is regular, for all $t\in T_{n,B}$. 
\end{prop}

\begin{proof}  By Proposition \ref{p:intrans}, $G$ contains two orbits $O, O'$ of sizes $1$ and $n-1$, and acts $(k-1)$-homogeneously on $O'$. It is easy to see that $B$ witnesses $k$-et
if and only if it contains the unique element $x$ in $O$. 

Consider an element $t' \in \langle G,t\rangle$. If $t'$ has rank $k$ or $n$, it is regular, so assume it has rank smaller than $k$. If the image $B'$ of $t'$ contains $x$, then $B'$ witnesses $|B'|$-et, and  $t'$ is regular. 

If $B'$ does not contain $x$ then $x \notin \{x\}t^{-1}$, and $yt=x$ for some $y \ne x$. By $|B'|$-homogeneity, there exists a $g\in G$ mapping $B'$ to a subsection of the kernel of $t$ containing $y$.
Then $t'gt$ has rank $|B'|$ and its image contains $x$ and hence witnesses $|B'|$-et. Again, $t'$ is regular, and hence $\langle G,t\rangle$ is regular. \qed
\end{proof}

The next theorem fully solves the case of $k=2$. 
\begin{theorem}	\label{k=2}
Let $n\ge 4$, $X=\{1,\ldots,n\}$, let  $G\le S_n$ and $B\subseteq X$ be a $2$-set. The following are equivalent: 
\begin{itemize}
\item  $\langle G,t\rangle$ is regular, for all $t\in T_{n,B}$;
\item  $G$ has the $2$-et property and $B$ witnesses it. 
\end{itemize}
The possible sets $B$ are:
\begin{itemize}
\item $G$ fixes one point, say $a$, and is transitive on the remaining points, in which case any $B$ containing $a$ works; 
\item $G$ has two orbits on $X$, say $A_1$ and $A_2$, and is transitive on $A_1\times A_2$, in which case any $B$ intersecting both $A_1$ and $A_2$ works;  
\item $G$ is transitive on $X$ and has at least one connected orbital graph, in which case $B$ can be any edge in one of the connected orbital graphs. 
\end{itemize}
\end{theorem}

\begin{proof}	
The first part follows from Theorem \ref{aux1} and the observation that all constants are idempotent (and hence regular). The second part follows from Proposition \ref{prop2.4} and the observations after it.  
\qed
\end{proof}

The next theorem fully solves the case of $k=3$. 
\begin{theorem}	\label{k=3}
Let $n\ge 6$, $X=\{1,\ldots,n\}$, let  $G\le S_n$ and $B\subseteq X$ be a $3$-set. Then $\langle G,t\rangle$ is regular for all $t\in T_{n,B}$ if and only if $G$ has the $3$-et property and $B$ witnesses it. 
\end{theorem}

\begin{proof}	

If $\langle G,t\rangle$ is regular, for all $t\in T_{n,B}$, then by Theorem \ref{aux1}, $G$ must possess the $3$-et property and $B$ must witness it. 
Conversely, let $G$ be $3$-et and $B$ be a witnessing set. 

The groups with the $3$-et property either are non-transitive, or transitive. If $G$ is  non-transitive,  the result follows from Proposition \ref{p:regintrans}.

If $G$ is transitive, either it is primitive or imprimitive. In the latter case, by Proposition \ref{p:2blocks}, $G$ has two blocks of imprimitivity and $B$ has two points from one of the blocks and a point from the other. Say that the blocks are $A_1:=\{x,z,\ldots\}$ and $A_2:=\{y,\ldots\}$, and $B:=\{x,y,z\}$. As $n\ge 6$ it follows that $|A_1|=|A_2|\ge 3$. 

Note that such a group cannot be totally imprimitive. By Proposition \ref{p:2et} and its proof, it follows that that $G$ has the $2$-et property, and that this is witnessed exactly by the sections of $A_1,A_2$, in particular by $\{x,y\}, \{y,z\}$.
Thus $G$ satisfies the conditions of Lemma \ref{l:reg-1}. We want to show that it also satisfies the additional condition of part (\ref{e:reg}) of the lemma. 
Consider any $3$-partition $\mathcal{P}$, and single out a part $P$. Suppose that the parts other than $P$ have a section $\{a,b\}$ that is also a section of $A_1,A_2$, and hence witnesses $2$-et. Let $c \in P$, then either $\{a,c\}$ or $\{b,c\}$ lies in a block of imprimitivity and hence does not witness $2$-et. 

By Lemma \ref{l:reg-1}(\ref{e:reg}), all elements of rank $2$ in 
$\langle G, t\rangle$ are regular.
 As permutations, constant maps, and all maps of rank $3$ are regular (the latter by $3$-et), $\langle G, t\rangle$ is regular for all $ t\in T_{n,B}$.

Now, suppose that $G$ is primitive, satisfies the $3$-et property with witness $B$, and that $t \in T_{n,B}$. Then by \cite[Theorem 1.8]{ArCa}, $G$ also satisfies the $2$-ut property. It follows that all maps of rank $3$ or $2$ in $\langle G, t\rangle$ are regular, and hence that $\langle G, t\rangle$ is regular.
\qed
\end{proof}

The possible sets $B$ are:
\begin{itemize}
\item if $G$ is intransitive (so it has a unique fixed point), any $3$-set containing the fixed point;
\item if $G$ is transitive but imprimitive (so it has two blocks of imprimitivity), any set containing two points from one block and one from the other;
\item if $G$ is primitive, then any $3$-set witnessing $3$-et.
\end{itemize}

We next address the case $k=4$. We start with a list of negative results.

\begin{lemma}\label{l:non4reg-line}
Suppose that $G$ is a $4$-et group of degree $n$ of one of the following types:
\begin{enumerate}
\item $\psl(3,q) \le G\le \pgaml(3,q)$, $n=q^2+q+1$, where $q\ge 3$ is a prime power;
\item $G\le \agaml(2,q)$, $n=q^2$, where $q\ge 3$ is a prime power;
\item $\psu(3,q) \le G \le \pgamu(3,q)$, $n=q^3+1$, for $3\le q\le 16$.
\end{enumerate}
Let $B$ be a set witnessing $4$-et. Then there exists a $t \in T_{n,B}$ such that $\langle G,t\rangle$ is not regular.
\end{lemma}

\begin{proof}
All listed groups preserve Steiner systems of type $(2,l,n)$, namely those of projective lines, affine lines, and those induced by projective lines. We will refer to any block of such a system as a line. Our numerical constraints guarantee that each line has at least $3$ points and that there are at least $5$ lines.

If a set $B$ witnessing $4$-et exists, it must witness weak~$4$-et. Such a set consists of $3$ points $x_1,x_2,x_3$ lying on a line $L$, and an additional point $x_4 \notin L$, and thus satisfies the conditions of Lemma \ref{l:reg-1} with $\bar B=\{x_1,x_2,x_3\}$. Pick a point 
$y$ that does not lie on any line containing two points from $B$. Let $K$ be the line through $y$ and $x_4$, and  $K'$ the line through $x_1$ and $y$. Now define $t$ with image $B$ by $x_1t^{-1}=K\setminus\{y\}$, $x_2t^{-1}=K'\setminus\{y\}$, $x_3t^{-1}=\Omega\setminus(K\cup K')$, $x_4t^{-1}=\{y\}$. 

Then $t$ satisfies the conditions of Lemma \ref{l:reg-1}(\ref{e:notreg}) with $g$ the identity. Hence $t^2$ is not regular.
 \qed
\end{proof}

\begin{lemma} \label{l:non4reg-HS}
Suppose that $G$ is $HS$ or $HS:2$ with its action on $100$ points.
Let $B$ be a set witnessing $4$-et. Then there exists a $t \in T_{n,B}$ such that $\langle G,t\rangle$ is not regular.
\end{lemma}

\begin{proof}
The elements of $G$ act as automorphisms of the triangle-free, strongly regular Higman-Sims graph $\Gamma$. 
As it satisfies weak~$4$-et,  $B$ consists of a $2$-path $x_1-x_2-x_3$ and an additional point $x_4$ not adjacent to any other elements of $B$. 
In $\Gamma$, non-adjacent vertices have $6$ common neighbours, hence we may pick a vertex $y$ adjacent to $x_2$ and $x_4$, with $y\ne x_1, x_3$.  
Now define $t$ with image $B$ by $x_1t^{-1}=\{x_4\}$, $x_2t^{-1}=\{y\}$, $x_3t^{-1}=\{x_2\}$, $x_4t^{-1}=\Omega \setminus \{y,x_2,x_4\}$. 

Consider $t^2$. By construction, $t^2$ has image $\{x_1,x_3,x_4\}$ and kernel classes $\{y\}, \{x_2,x_4\}, \Omega \setminus \{y,x_2,x_4\}$. 
The vertices in the image of $t^2$ are pairwise non-adjacent, while every vertex in $\{x_2,x_4\}$ is adjacent to $y$.
Hence $t$ satisfies the conditions of Lemma \ref{l:reg-1}(\ref{e:notreg}), with $\bar B=\{x_1,x_3,x_4\}$, and $g$ being the identity. By the lemma, $t^2$ is not regular in  $\langle G,t\rangle$. \qed

%
\end{proof}

\begin{theorem}Suppose that $G\le S_n$, $n\ge 8$, has $4$-et, and that $B$ witnesses it. {\color{black}In addition assume that $G\ne Sz(32):5$}, with $n=1025$. Then 
$\langle G,t\rangle$ is regular for every $t \in T_{n,B}$, if and only if $G$ is intransitive, $G=\agl(1,13)$ ($n=13$), or $3$-ut.
\end{theorem}

\begin{proof} The list of groups satisfying (or potentially satisfying) $4$-et is given in Theorem \ref{th:4-et}. If $G$ is intransitive, the results follows from Proposition \ref{p:regintrans}, and if $G$ has the $3$-ut property from Theorem \ref{semimain}. For $G=\agl(1,13)$, we have checked by computer that 
$G$ satisfies the conditions of Lemma \ref{l:reg-1}(\ref{e:reg}) for $B$ in both orbits witnessing $4$-et. By the lemma, all element of order $3$ in $\langle G,t\rangle$ are regular. As $\agl(1,13)$ is also $2$-ut, $\langle G,t\rangle$ is regular.
All remaining groups listed in Theorem \ref{th:4-et} are excluded by Lemmas \ref{l:non4reg-line} and \ref{l:non4reg-HS}. \qed \end{proof}
Concretely, the pairs $(G,B)$ introducing regularity in this way are those satisfying the following conditions (in the last two cases, only if $G$ has the $4$-et property). If no set $B$ is given, then either not all witnesses are known, or we could not find a suitable geometric description. 
\begin{enumerate}
\item $G$ fixes a point $x$ and acts $3$-homogeneously on $\Omega\setminus\{x\}$, $B$ is any $4$-set containing $x$;
\item $G=\agl(d,2)$, $d\ge 3$, $n=2^d$, $B$ is any affine independent $4$-set;
\item $G$ is one of $\agl(1,8)$, $\agaml(1,8)$ ($n=8$), $2^4.A_7$ ($n=16$), $\agaml(1,$ $32)$ $(n=32$), $B$ is a $\gf(2)$-affine independent $4$-set;
\item $G$ is one of 
$\agl(1,11)$ ($n=11$),  $\agl(1,13)$ ($n=13)$,   $2^6: G_2(2)$,  $2^6: \psu(3,3)$ ($n=64$), or $\Sz(8).3$ ($n=65$);
\item $G=2^d: \Sp(d,2)$, $d \ge4$ and even ($n=2^d$), $B$ is a mixed $4$-set;
\item $\psl(2,q)\le G \le \pgaml(2,q)$ for prime powers $q$ with $7 \le q \le 49$ ($n=q+1$);
\item $G=\Sp(2d,2)$, $d \ge 3$, in either of its $2$-transitive representations
($n=2^{2d-1}\pm2^{d-1}$), $B$ is a mixed $4$-set;
\item $G$ is one of $\psl(2,11)$ ($n=11$), $M_{22}$, $M_{22}:2$ ($n=22$), $B$ is not contained in any line/block of its biplane geometry/Steiner system $S(3,6,22)$;
\item $G=Co_3$ ($n=276$), $B$ is a mixed $4$-set;
\item $G$ is $4$-transitive or one of $\psl(2,8)$, $\pgaml(2,8)$ ($n=9$), $M_{11}$ ($n=12$), $\pgaml(2,32)$ ($n=33$), or $\pgaml(2,128)$ ($n=129$), $B$ is any $4$-set;
\item $\psl(2,q) \le G\le \pgaml(2,q)$, $G \ne \pgaml(2,128)$, for some prime power $q \ge 51$ ($n=q+1$); 
\item $G \in\{Sz(8) (n=65), HS (n=176) \}$.
\end{enumerate}

\begin{lemma}\label{l:pgl217k5} Let $G=\pgl(2,17)$ ($n=18$), and $B$ a set witnessing $5$-et. Then 
$\langle G,t\rangle$ is regular for every $t \in T_{18,B}$.
\end{lemma}
\begin{proof}
The group $G=\pgl(2,17)$ is $3$-ut, has one orbit not witnessing $4$-et, and $2$ orbits that witness $5$-et. The witnessing sets from one of these orbits contain $4$ subsets that witness $4$-et. In the other orbit there are $3$ such subsets. A computerised search (similar to the one described after Lemma \ref{l:reg-1}) shows that for every $5$-partition, there exists 
at least $3$ different collections of $4$ parts that each contain a section not witnessing $4$-et. 

If $t' \in \langle G,t\rangle$ of rank  $4$ has an image witnessing $4$-et, then it is regular. So assume otherwise. In the image of $t$, at least $3$ subsets
witness $4$-et. Applying the result of our computer search to the kernel of $t$, we see at least one of those $4$-et witnesses is the image of a $4$-set  $B'$ not witnessing $4$-et. As there is only one orbit of non-witnesses, there exist $g \in G$  mapping the image of $t'$ to $B'$. Thus $t'gt$ has an image that witnesses $4$-et,  and  $t'$ is regular. 

As $G$ possesses the $3$-ut property, the result follows. \qed
\end{proof}
\begin{lemma}\label{l:pgl225} Let $G=\pgl(2,25)$, $\pxl(2,25)$, or $\pgaml(2,25)$ ($n=26$) and $B$ a set witnessing $5$-et. Then 
$\langle G,t\rangle$ is regular for every $t \in T_{26,B}$.
\end{lemma}
\begin{proof} The group $G$ preserves a circle geometry with circles of size $6$. It is also $3$-ut, hence it suffices to consider $t' \in \langle G,t\rangle$ of rank $4$. If the image of $t'$ witnesses $4$-et, then $t'$ is regular. Otherwise the image belongs to one of two orbits $O, O'$ on $4$-sets. One orbit, say $O$, consists of $4$-subsets of circles. The witnessing set $B$
contains exactly one member of $O, O'$ each, and $3$ additional $4$-subsets that witness $4$-et.

Assume first that the image of $t'$ lies in $O'$, and consider the collection  $B^*$ of one-element subsets of $B$. 
By a similar computation as described after Lemma \ref{l:reg-1}, we have confirmed that it is not possible to extend $B^*$ to a $5$-partition of $\Omega$ in which every $4$-subsection from $O'$ lies in those  parts whose intersection with $B$ does nor witness $4$-et. Hence any partition has at least $2$ $O'$-subsections and one $O$-subsection that pairwise intersect different parts. 

We now apply this result to the kernel of $t$. If one of the $O'$-subsection in the kernel has an image that witnesses $4$-et, then for suitable $g \in G$, $t'gt$ has the same witnessing image, and $t'$ is regular. Otherwise, the kernel has only two such subsections, which map to elements of $O$ and $O'$, respectively. 
Hence there exist $g_1 \in G$, such that $t'gt$ has an image in $O$.
Moreover in this case,  there is an $O$-subsection that is 
mapped to an image that witnesses $4$-et. So for suitable $g_2\in G$, $t'g_1tg_2t$ also has this image, and $t'$ is regular.

Now let $t'$  have an image in $O$. A similar search reveals that any $5$-partition will either have at least $4$-subsection from $O$ that intersect different parts, or  consists of $4$ parts that partition a circle and one part containing the remaining elements. If the kernel of $t$ belongs to the first case, there exists $g \in G$ such that the image of $t'gt$ 
either lies in $O'$ or witnesses $4$-et. In the later case, $t'$ is regular. In the first case, we can repeat the above argument to show that $t's$ has an image witnessing $4$-et, and hence $t'$  is regular as well. 

Finally, if the kernel $t$ consists of the $4$-partition of a circle and an additional part, then in order for $t'$ to have an image in $O$, the $4$-subset of $B$ in $O$ cannot be the image of the classes that partition the circle. However, in this case, there exists $g \in G$ that maps the image of $t'$ to a section of the classes partitioning the circle. Thus $t'gt$ has an image that witnesses $4$-et or lies in $O'$, and the result follows as above. \qed
\end{proof}
\begin{lemma}\label{l:pgl227}
Let $G=\pgl(2,27)$ or $G=\pgaml(2,27)$ ($n=28$), and $B$ a witness for $5$-et. Then there exist $t \in T_{28,B}$ such that $\langle G, t\rangle$ is not regular.
\end{lemma}
\begin{proof}
In GAP, $G=\pgl(2,27)$ and $G=\pgaml(2,27)$ are both represented on the set $\{1,2,\dots, 28\}$. With regard to this representation, let $t$ be given by 
$$t^{-1}(5)= \{5, 3, 15, 25\}, t^{-1}(13) =\{8, 7, 14, 19\}, t^{-1}(18)=\{23\}, t^{-1}(19)=\{10\},$$ 
$$t^{-1}(23)=\{22, 1, 2, 4, 6, 9, 11, 12, 13,
16, 17, 18, 20, 21, 24, 26, 27, 28\}.
$$
Then $t$ can be checked (for both groups) to satisfy the conditions of Lemma \ref{l:reg-1}(\ref{e:notreg}) with (in the notation of the lemma) $g$ the identity and $\{10\}$ the kernel class mapped  to 
$B \setminus \bar B$. Hence $t^2$ is not regular in  $\langle G, t\rangle$. As $G$ has only one orbit witnessing $5$-et, the result follows.
\qed
\end{proof}

\begin{theorem}Suppose that $G\le S_n$, $n\ge 10$, has $5$-et, and that $B$ witnesses it. Then 
$\langle G,t\rangle$ is regular for every $t \in T_{n,B}$ if and only if $G \ne \pgl(2,27), \pgaml(2,27)$.
\end{theorem}

\begin{proof} The list of groups satisfying (or, in the case of $\pgaml(2,128)$, potentially satisfying) $5$-et is given in Theorem \ref{th567}. If $G$ is intransitive, the results follows from Proposition \ref{p:regintrans}, if $G$ is $4$-ut from Theorem \ref{semimain},  if $G=\pgl(2,17)$, from Lemma \ref{l:pgl217k5}, if $G=\pgl(2,25)$, 
$\pxl(2,25)$, or $\pgaml(2,25)$, from Lemma \ref{l:pgl225}, and if $G=\pgl(2,27)$ or $\pgaml(2,27)$, 
from Lemma \ref{l:pgl227}.

In all remaining cases, we have checked by computer that the groups satisfy the conditions of Lemma \ref{l:reg-1}(\ref{e:reg}). As these groups are also $3$-ut, the results follows. \qed \end{proof}

That is, the groups $G$ introducing regularity in this way are those satisfying the following conditions:
\begin{enumerate}
\item $G$ fixes one point and acts $4$-homogeneously on the remaining ones, and
$B$ contains the fixed point;
\item $G $ is one of  $\psl(2,11)$, $M_{11}$, $\pgl(2,11)$ ($n=12$), $\pgl(2,13)$ ($n=14$),  $\pgl(2,17)$ ($n=18$), $\pgl(2,$ $19)$ ($n=20$), $\pgl(2,23)$ ($n=24$), 
$\pgl(2,25)$,
$\pxl(2,25)$, $\pgaml(2,25)$ ($n=26$), $\psl(2,32)$ ($n=33$);
\item $G$ is on of $M_{10}$, $\pgl(2,9)$, $\pgaml(2,9)$ ($n=10$), $\psl(2,11)$ ($n=11$),  $\psl(2,16)$, $\psl(2,16):2$, $\pgaml(2,16)$ ($n=17$), $M_{22}$, $M_{22}:2$ ($n=22$), and $B$ contains exactly $4$ points from a circle/line/block of its circle geometry/biplane geometry/Steiner system $S(3,l,n)$.
\item $G$ is one of $M_{11}$ ($n=11$), or $M_{23}$ ($n=23$), and $B$ is not contained in or equal to a block of the Steiner system $S(4,l,n)$;
\item $G$ possesses $5$-ut, and hence is alternating, symmetric or one of
$M_{12}$ ($n=12$), $M_{24}$ ($n=24$), or $\pgaml(2, 32)$ ($n=33$), and $B$ is arbitrary;
\item $G=\pgaml(2,128)$ ($n=129$), provided that it satisfies $5$-et.
\end{enumerate}

\begin{lemma}  \label{l:regAGL42}Let $G=\agl(4,2)$ or $G=2^4:A_7$ ($n=16$), and $B$ a set witnessing $6$-et. Then there exists a $t\in T_{16,B}$ such that $M=\langle G,t\rangle$ is not regular.
\end{lemma}
\begin{proof}
For either group, $B$ consists of $5$ affine independent points plus a point forming a plane with $3$ of the other elements. Say $p_1,p_2,p_3,p_4 \in B$ form a plane, and 
$q,q'\in B$ are the additional points. Moreover, $G$ acts transitively on those $5$-sets that contain an affine plane.  

Consider $t \in T_{16,B}$ that is the identity on $\{p_1,\dots,p_4,q\}$, maps $q'$ to $q$, and all additional elements to $q'$. 
Let $q''$ be the fourth element of the plane containing $p_3,p_4,q$, and  $g \in G$  map $\{p_1,\dots, p_4\}$ to $\{p_3,p_4,q,q''\}$ and $q$ to $p_2$.
Then $t'=t^2gt$ has image $B\setminus\{p_1\}$, and hence is affine independent. Its kernel consists of $4$ singletons forming an affine plane, and another kernel class containing the remaining elements. This kernel will not admit an affine independent section. Now if $I$ is any $5$-set of affine independent points, so will $Ig$, for any $g \in G$. Moreover,
$It$ will either be affine independent or have rank at most $4$. Hence $Is$ will satisfies one of these conditions, for any $s \in \langle G,t \rangle$. It follows that $t'st'$ 
has rank at most $4$, and so $t'$ is not regular. 
\qed
\end{proof}

\begin{lemma}\label{l:pgl217k6} Let $G=\pgl(2,17)$ ($n=18$) and $B$ a set witnessing $6$-et. Then $\langle G, t\rangle$ is regular, for each $t\in T_{18,B}$.
\end{lemma}

\begin{proof} The group $G$ has the $3$-et property, has one orbit on $4$-sets that fails to witness $4$-et, two orbits $O, O'$ on $5$-sets that fail to witness $5$-et, and one orbit that witnesses $6$-et.
Three $5$-subsets of $B$ do not witness $5$-et, with two of those belonging to the same orbit, say $O'$. In addition, three $4$-subsets of $B$ do not witness $4$-et.

It suffices to show regularity for the $t' \in  \langle G, t\rangle$ of rank $4$ or $5$, whose images do not witness $4$-et or $5$-et. For $t'$ of rank $5$ we use a series of computations similar to the one in Lemma \ref{l:pgl225}. Consider first that the image of $t'$ lies in the orbit $O$ that only contains one subset of $B$ (in GAP this orbit is represented by
[4, 6, 10, 13, 17] ). Computation shows that every $6$-partition of $\Omega$ contains at least two $O$- and two $O'$-subsections, which pairwise intersect different parts of the partition. Applying this to the kernel of 
$t$, we see as in Lemma \ref{l:pgl225} that for suitable $g, g_1,g_2 \in G$, the image of either $t'gt$ or $t'g_1tg_2t$ witnesses $5$-et, implying the regularity of $t'$. 

If the image of $t'$ belongs to the orbit $O'$, we can similarly confirm that for any $6$-partition, there are at least three $O'$-subsections that pairwise intersect different parts. 
Hence $t'gt$, for suitable $g \in G$ has an image that either witnesses $5$-et, or belong to $O$, which implies that $t'$ is regular. 

Finally, for $t'$ of rank $4$, we similarly checked that the subsections of the kernel of $t$ that do not witness $4$-et include one whose image under $t$ does witness $4$-et. The result follows.  \qed
\end{proof}
\begin{lemma}\label{l:pgaml227k6} Let $G =\pgaml(2,27)$ ($n=28$), and $B$ witness $6$-et. Then there exists $t\in T_{28,B}$ such that $\langle G, t\rangle$ is not regular.
\end{lemma}
\begin{proof} The group $G$ preserves a circle geometry with circles of size $4$.  From this, it follows easily that $B$ contains exactly one circle $C=\{c_1,c_2,c_3,c_4\}$.
Let $e,d$ be the other elements of $B$. Moreover, as can be checked computationally, the $5$-sets containing exactly one circle form an orbit of $G$.

Let $f$ be the additional element in the circle containing $\{c_3,c_4,d\}$. Let $t\in T_{28,B}$ map $C$ identically, map $f$ to $e$ and every other element to $d$. As they lie in the same orbit on $5$-sets, there exists a $g\in G$ that maps $C\cup\{d\}$ to $\{c_2, c_3,c_4,d,f\}$. Consider $t'=t^2gt$. We claim the $t'$ is not regular in $\langle G, t\rangle$. 

Note first the if any $5$-set $S$ does not contain a circle, then neither do $St$ or $Sg$, for any $g\in G$. The image $\{c_2, c_3,c_4,d,e\}$ of $t'$ is such a set, and hence
the image of $t's$ is without circle as well, for any $s\in \langle G, t\rangle$. However, the kernel of $t'$ has $4$ singleton sets corresponding to the elements of $C$. 
It follows that $t'st'$ has rank at most $4$, for any $s \in \langle G, t\rangle$, and so is not regular.\qed

\end{proof}

\begin{theorem}Suppose that $G\le S_n$, $n\ge 12$, has $6$-et, and that $B$ witnesses it.  Then 
$\langle G,t\rangle$ is regular for every $t \in T_{n,B}$, if and only if $G$ is intransitive, $\pgl(2,17)$ ($n=18$), $M_{11}$ ($n=12$), $M_{23}$ ($n=23$), or $5$-ut.
\end{theorem}

\begin{proof} The list of groups satisfying $6$-et is given in Theorem \ref{th567}. If $G$ is intransitive, the results follows from Proposition \ref{p:regintrans}, and if $G$ is $5$-ut from Theorem \ref{semimain}. If $G=\agl(4,2)$ or $2^4:A_7$, the result follows from Lemma \ref{l:regAGL42}, if $G=\pgl(2,17)$,
from Lemma \ref{l:pgl217k6}, and if $G=\pgaml(2,27)$, from Lemma \ref{l:pgaml227k6}. For $G=M_{11}$ ($n=12$), or $G=M_{23}$, we have checked by computer that 
$G$ satisfies the conditions of Lemma \ref{l:reg-1}(\ref{e:reg}). As these groups are also $4$-ut, the result follows. \qed \end{proof}

That is, the groups $G$ introducing regularity in this way are those satisfying the following conditions:
\begin{enumerate}
\item $G$ fixes one point and acts $5$-homogeneously on the remaining ones, and
$B$ contains the fixed point;
\item $G $ is one of $\pgl(2,17)$ ($n=18$), $\pgaml(2,32)$ ($n=33$); 
\item $G$ is one of $M_{11}$, $M_{12}$ ($n=12$), $M_{24}$ ($n=24$), and $B$ not is contained or equal to a block of the Steiner system $S(5,l,n)$ preserved by $G$;  
\item $G=M_{23}$ ($n=23$), and $B$ contains exactly $5$ points from one block of the Steiner system $S(4,7,23)$;
\item $G$ is $6$-homogeneous and hence alternating or symmetric, and $B$ is
arbitrary.
\end{enumerate}

\begin{theorem}Suppose that $G\le S_n$, $n\ge 14$, has $7$-et, and that $B$ witnesses it. 
Then 
$\langle G,t\rangle$ is regular for every $t \in T_{n,B}$.
\end{theorem}

\begin{proof} By Theorem \ref{th567}, the $G\ne M_{24}$ satisfying $7$-et are either intransitive or $7$-homogeneous. If $G$ is intransitive, the results follows from Proposition \ref{p:regintrans}, and if $G$ is $7$-homogeneous and hence $6$-ut from Theorem \ref{semimain}.
So assume that $G=M_{24}$. Recall that $G$ preserves a Steiner system with parameters $(5,8,24)$, and that $B$ witnesses $7$-et if there is a block of the system containing exactly $6$ points of $B$. Moreover, $G$ has two orbits on $6$-sets consisting of those sets that are contained in a block or not, with the later witnessing $6$-et. Hence 
$G$  satisfies the condition of Lemma \ref{l:reg-1}.

The Steiner system has the property that any $7$-set contains $6$ points from a block. From this it follows easily that $G$ satisfies the conditions of Lemma \ref{l:reg-1}(\ref{e:reg}), 
and hence all rank $6$ elements in $\langle G,t\rangle$ are regular.
As $G$ possesses the $5$-ut property,  $\langle G,t\rangle$ is regular, for all $t \in T_{24,B}$.
 \qed \end{proof}

That is, the groups $G$ introducing regularity in this way are those satisfying the following conditions:
\begin{enumerate}
\item $G$ fixes one point and acts $6$-homogeneously on the remaining ones,
and $B$ contains the fixed point;
\item $G=M_{24}$ ($n=24$), and $B$ consists of seven points not in a block;
\item $G$ is $7$-homogeneous and hence alternating or symmetric, and $B$ is
arbitrary.
\end{enumerate}

\section{Problems}

We give here some problems to encourage further research on this topic.

\begin{problem}
Settle the remaining cases in Theorems~\ref{th567} and~\ref{th:4-et}. That is,
\begin{enumerate}
\item decide the 4-et property for the following groups:
\begin{itemize}
\item $n=q+1$: $\psl(2,q) \le G \le \pgaml(2,q)$, $G\ne \pgaml(2,128)$ for $q\ge 51$ a prime power;
\item $n = 65$: $\Sz(8)$;
\item $n = 126$: $\pgu(3, 5)$, $\pgamu(3, 5)$;
\item $n = 176$: $HS$;
\item $n = 513$: $\psu(3, 8).3$, $\psu(3, 8).6$, $\psu(3, 8).2^3$, $\pgaml(3,8)$;
\item $n = 730$: $\pgamu(3, 9)$;
\item $n = 1025$: $\Sz(32) : 5$;
\item $n = 4097$: $\pgamu(3, 16)$. 
\end{itemize}
\item decide the $5$-et property for the following group: $n = 129$: $\pgaml(2, 128)$.
\end{enumerate}
\end{problem}

\begin{problem}
There is a dual concept to the et property. We say that the permutation group
$G$ has the \emph{dual $k$-et property} with \emph{witnessing $k$-partition
$\mathcal{P}$} if, for every $k$-set $A$, there exists $g\in G$ such that
$Ag$ is a section for $\mathcal{P}$. Which groups have this property?
\end{problem}


\begin{problem}
Which groups $G$ have $k$-et for $k>n/2$? When is it the case that
$\langle G,t\rangle$ is regular for all $t$ whose image is a witnessing set?
\end{problem}

Let $\Omega$ be a finite set. We say that a set $\Sigma$ of $k$-subsets of $\Omega$ dominates a
set $\Pi$ of $k$-partitions of $\Omega$ if for every $\mathcal{P}\in\Pi$ there exists $S\in\Sigma$ such that $S$
is a transversal of $\mathcal{P}$. Similarly, we say that $\Pi$ dominates $\Sigma$ if given any set
$S\in\Sigma$ there exists $\mathcal{P}\in\Pi$ such that $S$ is a transversal for $\mathcal{P}$. Many arguments
in the classification of $k$-et groups would certainly be very simplified if the
answer to the following purely combinatorial questions was known.

\begin{problem}
Let $\Omega$ be a finite set and let $k\le |\Omega|/2$. Let $K$ be the set of all
$k$-subsets of $\Omega$ and let $P$ be the set of all $k$-partitions of $\Omega$.
\begin{enumerate}
\item Find the minimum of the set
\[\{|\Sigma| \mid \Sigma\subseteq K \hbox{ and } \Sigma \hbox{ dominates }P\}.\]
\item Find the minimum of the set
\[\{|\Pi| \mid \Pi\subseteq P \hbox{ and }\Pi \hbox{ dominates }K\}.\]
\end{enumerate}
\end{problem}

For non trivial bounds on (a) please see \cite{BT}. Assuming (b) is very difficult
too, at least provide some non trivial bound.

Paper \cite{ArMiSc} immediately prompts the following problem.

\begin{problem}
Classify the permutation groups on a finite set $\Omega$ that satisfy the following property: there exists $B\subseteq\Omega$ such that for all
transformations $t$ on $\Omega$ with image $B$, the semigroup
$\langle G, t\rangle \setminus G$ is idempotent generated.
\end{problem}

There are linear versions of these problems that we generally recall here (for more details and extensions to independence algebras please see \cite{ArCa}). 

\begin{problem}
Let $V$ be a finite dimension vector space over a finite field.  Classify the linear groups $G\le \Aut(V)$ such that for all linear transformations $t\in \End(V)$ the semigroup $\langle G,t\rangle$ is regular. If this problem could be solved it would yield the linear analogue of the main result in \cite{ArMiSc}.  

Find linear analogous to the main results in this paper and in \cite{ArCa}. 
\end{problem}

\section*{Acknowledgements} 
The authors would like to thank Jo\~ao Pedro Ara\'ujo (University  of Lisbon) for his help automating some computations, and Markus Pfeiffer (University of St Andrews) for his help with the computation confirming the $6$-et property for $\pgaml(2,32)$.

The first author was partially supported by the Funda\c{c}\~ao para a Ci\^encia e a Tecnologia
(Portuguese Foundation for Science and Technology)
through the project
CEMAT-CI\^ENCIAS UID/Multi/04621/2013, and through project ``Hilbert's 24th problem'' (PTDC/MHC-FIL/2583/2014).
The second author was supported by travel grants from the University of Hull's Faculty of Science and Engineering and the Center for Computational and Stochastic Mathematics.

\end{document}